\documentclass[11pt]{article}
\usepackage{graphicx} 
\usepackage[utf8]{inputenc}
\usepackage[margin=1in]{geometry}
\usepackage{amssymb, amsmath, amsthm}
\usepackage{mathtools}
\usepackage{graphicx}
\usepackage{enumitem}
\usepackage{tikz}
\usepackage{mathrsfs}
\usepackage{comment}
\usepackage[T1]{fontenc}
\usepackage{hyperref}% Required for inserting images
\theoremstyle{plain}
\newtheorem{lemma}{Lemma}[section]

\newtheorem{theorem}[lemma]{Theorem}
\newtheorem{proposition}[lemma]{Proposition}
\newtheorem{proposition-definition}[lemma]{Proposition-Definition}

\newtheorem{definition}[lemma]{Definition}

\newtheorem{theorem-definition}[lemma]{Theorem-Definition}
\newtheorem{remark}[lemma]{Remark}
\newtheorem{example}[lemma]{Example}

\renewcommand{\P}{\mathcal P}

\newcommand{\D}{\mathcal D}

\newcommand{\Z}{\mathbb Z}

\newcommand{\R}{\mathbb R}
\newcommand{\C}{\mathbb C}
\newcommand{\N}{\mathbb N}
\newcommand{\E}{\mathbb E}

\newcommand{\Luk}{\mathbf L}
\newcommand{\setzeroes}{x_1=\cdots=x_N=0}
\newcommand{\la}{\lambda}
\newcommand{\M}{\mathcal M}
\renewcommand{\d}{\textrm d}
\newcommand{\MM}{\mathbf M}
\newcommand{\vecx}{\vec{x}}
\newcommand{\zeroes}{\vec{x}=(0^N)}
\newcommand{\proj}{\text{proj}}
\newcommand{\corner}{\text{corner}}

\newcommand{\bfp}{\mathbf{p}}
\newcommand{\bfq}{\mathbf{q}}

\title{Law of Large Numbers for continuous $N$-particle ensembles at fixed temperature}

\author{Cesar Cuenca and Jiaming Xu}

\date{\today}

\begin{document}

\maketitle

\begin{abstract}
    In this paper, we find necessary and sufficient conditions for the Law of Large Numbers of averaged empirical measures of $N$-particle ensembles, in terms of the asymptotics of their Bessel generating functions, in the fixed temperature regime. This settles an open problem posed by Benaych-Georges, Cuenca and Gorin. For one direction, we use the moment method through Dunkl operators, and for the other we employ a special case of the formula of Chapuy--Dolega for the generating function of infinite constellations. As applications, we prove that the LLN for $\theta$-sums and $\theta$-corners of random matrices are given by the free convolution and free projection, respectively, regardless of the value of inverse temperature parameter $\theta$. We also prove the LLN for a time-slice of the $\theta$-Dyson Brownian motion.
\end{abstract}

\tableofcontents

\section{Introduction}

\subsection{Overview}

In a series of papers around the year 1990, Voiculescu laid the foundations of what would be known as the theory of free probability, from the viewpoint of operator algebras; see e.g.~\cite{Voi86}.
The subject soon became of main interest in probability theory, after unexpected connections with random matrix theory were discovered, such as the phenomenon of \emph{asymptotic freeness}~\cite{Voi91}.
This principle roughly states that large independent Hermitian random matrices behave like free random variables. In particular, the empirical spectral distribution of sums of independent Hermitian random matrices tends to the free convolution of measures, as the matrix dimensions tend to infinity.

By restricting our attention to the spectra of Hermitian matrices, the operation of adding two matrices admits a deformation depending on the real parameter $\beta>0$, in such a way that $\beta=1,2,4$ correspond to adding self-adjoint real, complex, and quaternionic matrices, respectively.
In this paper, however, we will employ the deformation parameter $\theta=\beta/2$, which is more suitable from the point of view of multivariate special functions.
The resulting operation is then called \emph{$\theta$-addition of matrices}, though in fact it is an operation on sequences of real numbers (the eigenvalues).
One can similarly define $\theta$-projection and $\theta$-multiplication, as the deformations of the operations of cutting corners and multiplying matrices, respectively (see \cite{GM20} for the precise definitions of these three operations).
These deformations are natural from the point of view of statistical mechanics, which treats random matrix eigenvalues as mutually repellent random particles and  $\beta$ as the inverse temperature, e.g.~$\beta$-ensembles and log-gas systems. We refer the interested reader to~\cite{Forrester} for a standard reference on these ideas.

The Law of Large Numbers (LLN) of $N$-particle systems that characterizes the asymptotic global behavior of the empirical measures has been a central topic in random matrix theory and integrable probability for many years.
The idea of using multivariate generating functions of representation theoretic origin on these problems dates back at least to~\cite{BuG15}, which studies discrete $N$-particle systems with the help of Schur generating functions.
This was later generalized to Jack generating functions with general parameter $\theta>0$ in~\cite{Hua21}.
In these articles, it was shown that the limits of empirical measures for $\theta$-additions of matrices can be described by the operation of quantized free convolution, instead of free convolution, in the regime of fixed temperature.
In more recent literature, the high-temperature limit regime, where $\theta\to 0$ at the same time as $N\rightarrow\infty$, so that $N\theta\rightarrow\gamma>0$, has also been a significant source of interest.
This double scaling regime was studied, for example, in \cite{CuD25(1), CuD25(2), CDM} and \cite{BCG} for discrete and continuous settings, respectively. 

In this paper, we prove the LLN for the $\theta$-sum of matrices, in the regime when $\theta>0$ is fixed and the number $N$ of eigenvalues (the matrix dimension) tends to infinity.
We also prove the LLN for the $\theta$-projection of matrices, in the same limit regime.
In fact, in our main theorem, we find necessary and sufficient conditions for the LLN of sequences of $N$-tuples in terms of the corresponding Bessel generating functions; this result answers an open problem posed by Benaych-Georges, Cuenca and Gorin in~\cite{BCG}.
We will then see that the results for $\theta$-sums and $\theta$-corners are simple corollaries.
We will also consider the $\theta$-Dyson Brownian motion started at arbitrary initial conditions.

We point out that our main \emph{if and only if LLN theorem} (namely, Theorem~\ref{thm:main}) has been independently proven by Yao~\cite{Y}.
However, his proof is completely different in one of the directions.
In fact, a new ingredient in our proof that the LLN implies an analytic condition on the sequence of Bessel generating functions of the measures is the recently-discovered topological expansion of Chapuy--Dolega~\cite{CD}.
This relates the combinatorial objects known as infinite constellations to noncrossing set partitions, or Łukasiewicz paths.

\subsection{Main result}

The Multivariate Bessel Functions (MBFs) are certain entire functions on $2N$ complex variables $a_1,\dots,a_N,x_1,\dots,x_N$, which depend on the inverse temperature parameter $\theta>0$.
They will be denoted by $B_{(a_1,\dots,a_N)}(x_1,\dots,x_N;\theta)$ and are defined as the normalized, joint symmetric eigenfunctions of Dunkl operators; see Section~\ref{sec:mbf} for details.

For the purposes of stating the main result in this introduction, we do not need an exact formula for the MBF.
We simply point out that if $N=1$, the MBF is $B_{(a)}(x;\theta)=e^{ax}$, while if $\theta\to 0$, the MBF becomes
\begin{equation}\label{eq:special_case}
\lim_{\theta\to 0}{B_{(a_1,\dots,a_N)}(\vecx;\theta)} = \frac{1}{N!}\sum_{\sigma\in S_N}{ e^{a_1x_{\sigma(1)} + a_2x_{\sigma(2)} + \,\dots\, + a_Nx_{\sigma(N)}} },
\end{equation}
where $\vecx:=(x_1,\dots,x_N)$. Also, for $\theta>0$ and the special case when $a_1<\dots<a_N$, the MBF admits an $\frac{N(N-1)}{2}$-dimensional integral representation; see~\cite[Prop.~3.3]{Cue}.

Then, for any probability measure $\mu_N$ on the closed Weyl chamber $\overline{\mathcal{W}_N} = \{ (a_1,\dots,a_N)\in\R^N \mid a_1\le \dots\le a_N \}$ satisfying the technical condition~\eqref{eq:technical} below, we will associate the following analytic function $G_\theta(\vecx;\mu_N)$ in a neighborhood of $\big(0^N\big)\in\C^N$:
    \[
    G_\theta(\vecx;\mu_N) := \int_{ a_1\le  a_2\le\dots\le a_N} B_{( a_1,\dots, a_N)}(\vecx;\theta)\mu_N(\d a_1,\dots,\d a_N).
    \]
This is called the \emph{Bessel generating function of $\mu_N$}.
In view of equation~\eqref{eq:special_case}, this should be regarded as a symmetric one-parameter $\theta$-deformation of the Fourier transform of $\mu_N$.

\begin{theorem}[See Theorem~\ref{thm:main} in the text for details]\label{thm:intro}
    The following two statements are equivalent, regarding a sequence $\{\mu_N\}_{N\ge 1}$, where each $\mu_N$ is a probability measure on $\R^N$.
    \begin{itemize}
        \item Let $a(N)=(a_1(N),...,a_N(N))\in\R^N$ be $\mu_N$-distributed, for each $N\in\Z_{\ge 1}$. There exist real numbers $m_1,m_2,\dots$ such that, for any $s\in\Z_{\ge 1}$ and $k_1,...,k_s\in \Z_{\ge 1}$, we have
    \begin{equation}\label{eqn:intro_convergence}
    \lim_{N\rightarrow\infty}{\E_{\mu_N}\left[\prod_{j=1}^s\left(\frac{1}{N}\sum_{i=1}^N a_i(N)^{k_j}\right)\right]} = \prod_{j=1}^s{m_{k_j}}.
    \end{equation}
    If this occurs, we shall 
    write:
    $$\lim_{N\rightarrow\infty}{a(N)} \stackrel{m}{=} \{m_k\}_{k\ge 1}.$$

        \item For any partition $\la$ with $\ell(\la)\le N$, let $a_\la^N$ be the coefficient of $p_\la(\vecx)$ in the expansion of $\ln\big(G_\theta(\vecx;\mu_N)\big)$ in the basis of power sum symmetric functions.
        Then there exist real numbers $\kappa_1,\kappa_2,\dots$ such that
    \begin{enumerate}[label=(\alph*)]
        \item $\displaystyle\lim_{N\to\infty}{\frac{a_{(d)}^N}{N} \bigg|_{\zeroes}} = \frac{\theta^{1-d}}{d}\cdot\kappa_d$, for all $d\in\Z_{\ge 1}$,
        \item $\displaystyle\lim_{N\to\infty}{\frac{a_\la^N}{N^{\ell(\la)}} \bigg|_{\zeroes}} = 0$, whenever $\ell(\la)\ge 2$.
    \end{enumerate}
    \end{itemize}
    If one (and therefore both) of the previous two equivalent statements holds, then $\{m_k\}_{k\ge 1}$ and $\{\kappa_\ell\}_{\ell\ge 1}$ are related to each other by the moment -- free cumulant formulas in equation~\eqref{eq_fixedtempmoment-cumulant}.
\end{theorem}

An important case of our theorem is when $\{m_k\}_{k\ge 1}$ is the sequence of moments of a unique probability measure $\nu$ on $\R$.
Then the convergence in Eqn.~\eqref{eqn:intro_convergence} implies that the empirical measures $\frac{1}{N}\sum_{i=1}^N{\delta_{a_i(N)/N}}$ converge weakly, in probability, to $\nu$.

\subsection{Applications: LLN for $\theta$-sums and $\theta$-corners}

We briefly outline here the applications of Theorem~\ref{thm:main}, to be discussed in more detail in Section~\ref{sec:application}.

\smallskip

Given two independent random $N\times N$ Hermitian matrices $A$, $B$ with uniformly random eigenvectors, and deterministic eigenvalues $a_1,\dots,a_N$ and $b_1,\dots,b_N$, respectively, the behavior of the (random) eigenvalues of the sum
$$C=A+B,$$
as $N\rightarrow\infty$, has been heavily studied in the literature; see e.g.~\cite{Voi91,B1,BCG}, and references therein.
For a general $\theta>0$, the one-parameter generalization of the above operation on eigenvalues, known as the $\theta$-addition, is defined as follows.

\begin{proposition-definition}
Given arbitrary $a=(a_1\le\dots\le a_N),\,b=(b_1\le\dots\le b_N)\in\R^N$, there exists a unique normalized, generalized function\footnote{Throughout this text, a \emph{generalized function} is a continuous linear functional on the set of compactly supported smooth functions.} on the Weyl chamber $\mathcal{W}_N:=\{c=(c_1\le\dots\le c_N)\}\subset\R^N$, denoted by $a+_\theta b$, such that
\begin{equation}\label{eq_thetaaddition1}
    B_{(a_1,\dots,a_N)}(\vecx;\theta)B_{(b_1,\dots,b_N)}(\vecx;\theta) = \int_{c=(c_1\le\dots\le c_N)} B_{(c_1,\dots,c_N)}(\vecx;\theta)\,\d_{a+_\theta b}(c),
\end{equation}
for all $\vecx\in\R^N$.
The generalized function above is compactly supported, so the right side of Eqn.~\eqref{eq_thetaaddition1} is well-defined.
Furthermore, $a+_\theta b$ can be equally defined for random $N$-tuples $a=(a_1\le\dots\le a_N)$, $b=(b_1\le\dots\le b_N)$, by the equality
\begin{equation}\label{eq_thetaaddition2}
\E_a\Big[ B_{(a_1,\dots,a_N)}(\vecx;\theta) \Big]
\cdot\E_{b}\Big[ B_{(b_1,\dots,b_N)}(\vecx;\theta) \Big]
= \E_{c=a+_\theta b} \Big[ B_{(c_1,\dots,c_N)}(\vecx;\theta) \Big],
\end{equation}
for all $\vecx\in\R^N$, whenever the distributions of $a$ and $b$ are exponentially decaying probability measures (see Definition \ref{def:bgf}), and $\E_{c=a+_\theta b} \Big[ B_{(c_1,\dots,c_N)}(\vecx;\theta)\Big]$ is again taken by testing the generalized function corresponding to $c$ against the multivariate Bessel function.

Moreover, the distribution of $c=a+_{\theta}b$ is uniquely determined by the equality above.
\end{proposition-definition}

It is a well-known conjecture that $a+_\theta b$ is not only a generalized function, but a probability measure; see e.g.~\cite{A} for more details. At the discrete level, this positivity conjecture is related to Stanley's conjecture of positivity of Jack--Littlewood--Richardson coefficients; see~\cite{S}. When $\theta=\frac{1}{2},1,2$, equation~\eqref{eq_thetaaddition2} recovers the usual additions of real, complex and quaternionic Hermitian random matrices. 

It is a classical result of Voiculescu~\cite{Voi91} that, if $\theta=\frac{1}{2},1$, and assuming that the empirical spectral measures of the independent $N\times N$  matrices $A(N)$ and $B(N)$ weakly converge to some probability measures $\mu_a$ and $\mu_b$, then the empirical measures of $C(N)=A(N)+B(N)$ also converge to some probability measure $\mu_c$. Moreover, this probability measure is the \emph{free convolution}
$$\mu_{c}=\mu_a\boxplus\mu_b,$$
that can be defined, for example, through the notion of free cumulants. We refer the reader to \cite{MS,NS} for standard references.

Our first application of Theorem \ref{thm:main} is the following generalization of Voiculescu's result, for all $\theta>0$. Due to the analytic subtlety of $a+_{\theta}b$, we state our result in terms of moments, i.e, testing the empirical spectral measures with polynomials instead of bounded continuous functions.

\begin{theorem}\label{thm:beta_addition}
    Let $a(N)=(a_1(N)\le\dots\le a_N(N))$, $b(N)=(b_1(N)\le\dots\le b_N(N))$, $N=1,2,\dots$, be two sequences of independent exponentially decaying random vectors. Suppose that there exist two sequences of real numbers $\{m_{k}^{a}\}_{k=1}^{\infty}$, $\{m_{k}^{b}\}_{k=1}^{\infty}$ such that
    \[
    \lim_{N\to\infty}{\frac{a(N)}{N}} \stackrel{m}{=} \{m_{k}^{a}\}_{k=1}^{\infty},\qquad
    \lim_{N\to\infty}{\frac{b(N)}{N}} \stackrel{m}{=} \{m_{k}^{b}\}_{k=1}^{\infty}.
    \]
    Then
    \[
    \lim_{N\to\infty}{\frac{a(N) +_\theta b(N)}{N}} \stackrel{m}{=} \{m_{k}^{a}\}_{k=1}^{\infty} \boxplus \{m_{k}^{b}\}_{k=1}^{\infty},
    \]
    where the right hand side is the free convolution of $ \{m_{k}^{a}\}_{k=1}^{\infty}$ and $\{m_{k}^{b}\}_{k=1}^{\infty}$.
\end{theorem}

Besides addition, another natural operation on an $N\times N$ Hermitian matrix $A$ with eigenvalues $(a_1\le\dots\le a_N)=:a^{(N)}$ is the projection on corners, i.e., the operation of taking the $M\times M$ principal submatrix $A^{(M)}$ consisting of the first $M$ rows and columns of $A$, for some $1\le M<N$.
We will consider the case where the eigenvalues of $A$ are deterministic and the eigenvectors are uniformly random as before; then the random $M$-tuple of real eigenvalues of $A^{(M)}$ will be denoted by $\big(a^{(M)}_1\le...\le a^{(M)}_M\big)=:a^{(M)}$.
We can do the same projection on corners for real symmetric and quaternionic Hermitian matrices.
In all cases, by considering all principal submatrices $A^{(M)}$, as $M$ ranges over $1,2,\dots,N$, we obtain a triangular array $\big\{a^{(M)}_i \colon 1\le i\le M\le N\big\}$ of random eigenvalues, where the top row $a^{(N)}$ is the fixed set of eigenvalues of $A$.
This random array is known as the \emph{corners process}, and a key feature of it is the \emph{interlacing condition} 
\[
a^{(k+1)}_1\le a^{(k)}_1\le a^{(k+1)}_2\le\dots \le a^{(k+1)}_k\le a^{(k)}_k\le a^{(k+1)}_{k+1},
\]
denoted by $a^{(k)}\prec a^{(k+1)}$, and valid for all $k=1,...,N-1$.

In all cases, the joint density function for the array of eigenvalues $\big\{a^{(M)}_i\big\}$ can be calculated explicitly.
It turns out that this formula can be naturally expressed as a function of a positive parameter $\theta$, such that for $\theta=\frac{1}{2},1,2$, it specializes to the cases of real symmetric, complex Hermitian and quaternionic Hermitian matrices, respectively.
Moreover, for a general $\theta>0$, this formula is the density of an actual probability measure on the set of interlacing arrays $a^{(1)}\prec\cdots\prec a^{(N-1)}\prec a^{(N)}$, where $a^{(N)}=a$ is deterministic.
This measure is called the \emph{$\theta$-corners process}; directly, it can be defined as the probability measure on the set of $\frac{N(N-1)}{2}$ interlacing coordinates $a^{(1)}\prec\cdots\prec a^{(N)}=a$, with density function
\begin{equation}\label{eq_thetacorners}
    \Lambda_{\theta}(\vec{y};a) := \frac{1}{Z_{N,\theta,a}}\prod_{k=1}^{N-1}\prod_{1\le i<j\le k}|y^k_i - y_j^k|^{2-2\theta}\prod_{p=1}^{k}\prod_{q=1}^{k+1}|y_{p}^{k}-y_{q}^{k+1}|^{\theta-1}\cdot\mathbf{1}_{y^{1}\prec y^{2}\prec\cdots \prec y^{N-1}\prec y^N=a},
\end{equation}
where $\vec{y}$ is the $\frac{N(N-1)}{2}$-dimensional vector with coordinates $\{y^k_i \colon 1\le i\le k\le N-1\}$, and where $Z_{N,\theta,a}$ is the normalization constant; see e.g.~\cite[Sec.~2]{Cue}.
We will assume the strict inequalities $a_1<\dots<a_N$, so that no particle $y^k_i$ is deterministic.

\begin{definition}
    For $1\le M< N$, $\theta>0$, and $a=(a_1<\dots<a_N)$ a random real $N$-tuple, the random $M$-tuple whose distribution is the marginal of $y^{(M)}$, according to the distribution $\Lambda_{\theta}(\vec{y};a)$, is denoted by $\text{corner}^{N}_{M}(a)$ and called the $M^{th}$ level marginal of the $\theta$-corners process of $a$.
\end{definition}

For a random $N$-tuple $a(N)=(a_1(N)<\dots<a_N(N))$, we denote its empirical measure by
\begin{equation}\label{eq:empirical_measure}
\mu[a(N)] := \frac{1}{N}\sum_{i=1}^{N}\delta_{\frac{a_{i}(N)}{N}}.
\end{equation}

\begin{theorem}\label{thm:beta_projection}
    Let $a(N)$, $N=1,2,\dots$, be a sequence of random tuples and let $\mu_a$ be a compactly supported probability measure such that
    \begin{equation}\label{assumption_corners}
    \lim_{N\rightarrow\infty}\mu[a(N)]\stackrel{m}{=}\mu_a.
    \end{equation}
    Then
    \[
    \lim_{N\to\infty}{\mu\big[\corner^N_{\lfloor\alpha N\rfloor}\big({a(N)}\big)\big]} \stackrel{m}{=} \proj_\alpha\big(\mu_a\big),
    \]
    where the right hand side is the free $\alpha$-projection of $\mu_a$, defined as the unique probability measure with free cumulants 
    $$\kappa_{d}^{(\alpha)} = \frac{1}{\alpha}\kappa_{d}, \quad\text{for all }d\ge 1.$$
\end{theorem}

\subsection{Methodology}

Theorem \ref{thm:main} has two directions, the ``if'' and ``only if'' parts; we take quite different approaches for their proofs.
In the ``if'' direction, the key is that we are able to extract the moments from the Bessel generating functions by means of the Dunkl operators.
While Dunkl operators were introduced in a purely algebraic background and historically play a role in the study of Calogero--Moser--Sutherland quantum many-body system, they have recently found many applications in random matrix theory and statistical mechanics; see~\cite{BCG,X,KX,GXZ,Y}.

One new feature in our argument, compared to the previous literature, is that we consider the expansions of logarithms of Bessel generating functions in the basis of power sum symmetric polynomials instead of monomial symmetric polynomials.
While the coefficients for the latter basis can be more easily extracted by taking partial derivatives, hence appearing like a more natural choice, their asymptotics do not give the right characterization of the LLN. 

In the ``only if'' direction, while it seems doable to reverse the arguments developed in the forward direction, we employed a different tool that greatly simplified the proof.
This is the new topological expansion, due to Chapuy--Dolega~\cite{CD}, of the $\theta$-generalized tau-function in terms of the combinatorial objects called \emph{constellations}, which are certain maps on closed orientable and non-orientable surfaces.
We take a specialization of this formula and identify the tau-function with the multivariate Bessel function; then the topological expansion highlights what the asymptotically relevant terms are.
It is worth noting that there have been other connections between topological expansions and random matrices; see, e.g.~\cite{BG13,CE06,BoGF}.
Although topological expansions arise in very different ways compared to our approach, there might be a deeper structural explanation lying behind their connection to random matrices; we hope to investigate this in future works.

\subsection*{Acknowledgments}

The authors are grateful to Houcine Ben Dali,  Maciej Do\l{}\k{e}ga and Vadim Gorin for helpful discussions.
CC was partially supported by the NSF grant DMS-2348139 and by the Simons Foundation’s Travel Support for Mathematicians grant MP-TSM-00006777.

\section{Bessel generating functions}

\subsection{Conventions}

In this paper, $N$ denotes a positive integer that will tend to infinity.
We will work with $N$ variables $x_1,\dots,x_N$ and let $\vec{x}=(x_1,\dots,x_N)$.
The partial derivatives will be denoted
\[
\partial_i := \frac{\partial}{\partial x_i},\qquad i=1,\dots,N,
\]
for simplicity, and the operator that permutes the variables $x_i$ and $x_j$, where $i\ne j$, will be denoted $\sigma_{i,j}$.
Moreover, $\theta>0$ will be a fixed parameter throughout the paper.

\subsection{Dunkl operators}

The Dunkl operators, introduced in \cite{D}, are defined by
\[
\D_i^{N,\theta} := \partial_i + \theta\sum_{j\colon j\ne i} \frac{1}{x_i-x_j}(1-\sigma_{i,j}),\qquad i=1,\dots,N.
\]
They pairwise commute, $\D_i^{N,\theta}\D_j^{N,\theta} = \D_j^{N,\theta}\D_i^{N,\theta}$, for all $i,j$, and their symmetric joint eigenfunctions, discussed in the next subsection, will play an important role in this paper.
We can view $\D_i^{N,\theta}$ as operators acting on holomorphic functions on domains of $\C^N$ that are symmetric with respect to permutations of the coordinates.
We will also use the symmetric operators
\[
\P_k^{N,\theta} := \sum_{i=1}^N{ \Big(\D_i^{N,\theta}\Big)^k },\qquad k=1,2,\cdots.
\]
For simplicity, we drop the superscripts $N,\theta$, so that $\D_i^{N,\theta}, \P_k^{N,\theta}$ will simply be denoted $\D_i, \P_k$.

\subsection{Multivariate Bessel functions}\label{sec:mbf}

\begin{theorem}[\cite{O}]\label{thm:bessel_existence}
    Let $a_1,\dots,a_N\in\C$ be arbitrary. There exists a unique holomorphic function $F(x_1,\dots,x_N)$, normalized by $F(0^N)=1$, that is symmetric with respect to its variables and solves the differential equations
    \[
    \P_k F(x_1,\dots,x_N) = \left( \sum_{i=1}^N{a_i^k} \right)F(x_1,\dots,x_N),\quad\textrm{for all }k\ge 1.
    \]
    If we denote this solution by $B_{(a_1,\dots,a_N)}(x_1,\dots,x_N;\theta)$, then the map $(a,x)\mapsto B_a(x;\theta)$ can be extended to an entire function on $2N$ variables.
\end{theorem}

Due to the uniqueness of Theorem~\ref{thm:bessel_existence}, it follows that $B_{(a_1,\dots,a_N)}(x_1,\dots,x_N;\theta)$ is also symmetric with respect to the variables $a_i$.
When $a_1<\dots<a_N$, the multivariate Bessel function admits the following integral representation (see \cite[Prop.~3.3]{Cue}):
\begin{equation}\label{eq_besselbranching}
B_a(x_1,\dots,x_N;\theta)=\int_{y^1\prec y^2\prec\dots\prec y^{N-1}}\exp\left(\sum_{k=1}^{N}x_{k}\left(\sum_{i=1}^{k}y^{(k)}_{i}-\sum_{j=1}^{k-1}y^{(k-1)}_{j}\right)\right)\Lambda(\vec{y};a)d\vec{y},
\end{equation}
where $\Lambda_{\theta}(\vec{y},a)$ is the $\theta$-corners process defined in Eqn.~\eqref{eq_thetacorners}, the $\frac{N(N-1)}{2}$--dimensional integral is over $k$-tuples $y^k=(y^k_1,\dots,y^k_k)$, for $k=1,2,\dots,N-1$, and the conditions $y^i\prec y^{i+1}$ mean that $y^{i+1}_1\le y^{i}_1\le y^{i+1}_2\le\dots y^{i+1}_i\le y^{i}_i\le y^{i+1}_{i+1}$.

\subsection{Bessel generating functions}

Following \cite{BCG}, let $\M_N$ be the set of Borel probability measures on ordered real $N$-tuples $a_1\le\dots\le a_N$ and, for any $R>0$, let $\M_N^R\subset\M_N$ be the subset of measures $\mu$ such that
\begin{equation}\label{eq:technical}
\int_{a_1\le\dots\le a_N} e^{NR\max_i|a_i|}\,\mu(\d a_1,\dots,\d a_N) < \infty.
\end{equation}

\begin{proposition-definition}[Lemma~2.9 of~\cite{BCG}]\label{def:bgf}
    If $\mu\in\M_N^R$ for some $R>0$, then the integral
    \[
    G_\theta(x_1,\dots,x_N;\mu) := \int_{ a_1\le\dots\le a_N} B_{( a_1,\dots, a_N)}(x_1,\dots,x_N;\theta)\mu(\d a_1,\dots,\d a_N)
    \]
    defines a holomorphic function in the domain
    \[
    \{ (x_1,\dots,x_N)\in\C^N \colon\ |\Re x_i|<R,\ i=1,2,\dots,N \},
    \]
    that is called the \textbf{Bessel generating function} of $\mu$.
\end{proposition-definition}

The measures in $\M_N^R$ will be called \emph{$R$-exponentially decaying}.
We also let $\M_N^{>0} := \bigcup_{R>0}{\M_N^R}$, and call all measures in $\M_N^{>0}$ \emph{exponentially decaying}.
Any $\mu\in\M_N^{>0}$ admits a Bessel generating function $G_\theta(x_1,\dots,x_N;\mu)$ that is a holomorphic function in a neighborhood of the origin.
By the normalization and symmetry of $B_{(a_1,\dots,a_N)}(x_1,\dots,x_N;\theta)$, it follows that
\[
G_\theta(0^N;\mu)=1,
\]
and $G_\theta(x_1,\dots,x_N;\mu)$ is symmetric with respect to $x_1,\dots,x_N$.

\begin{proposition}[Prop.~2.11 from~\cite{BCG}]\label{prop:action_dunkl}
    Let $s\ge 1$, $k_1,\dots,k_s\ge 1$ be any integers, and let $\mu\in\M_N^{>0}$. Then
    \[
    \left( \prod_{i=1}^s \P_{k_i} \right) G_\theta(x_1,\dots,x_N;\mu) \Big|_{x_1=\dots=x_N=0} = \mathbb{E}_\mu\left[ \prod_{i=1}^s \left( \sum_{j=1}^N (a_j)^{k_i} \right) \right],
    \]
    where, in the right hand side, $a_1\le\dots\le a_N$ is $\mu$-distributed.
\end{proposition}

\section{Statement of the main theorem}

Let $\{\mu_N\in\M_N\}_{N\ge 1}$ be a sequence of measures on real $N$-tuples $a=(a_1<\cdots<a_N)$.
The random variables of interest are
\[
p_k^N := \sum_{i=1}^N{\left(\frac{a_i}{N}\right)^k},\qquad k=1,2,\cdots.
\]

\begin{definition}\label{def:lln}
    We say that $\{\mu_N\in\M_N\}_{N\ge 1}$ \textbf{satisfies the fixed temperature LLN} if there exists a sequence of real numbers $\{m_k\}_{k\ge 1}$ such that
    \[
        \lim_{N\to\infty}{ \mathbb{E}_{\mu_N}\left[ \frac{1}{N^s}\prod_{j=1}^s{ p^N_{k_j}(a) } \right] } = \prod_{j=1}^s{m_{k_j}},
    \]
    for all integers $s\ge 1$, and $k_1,\dots,k_s\ge 1$.
\end{definition}

Assume that $\mu_N\in\M_N^{>0}$ and let
\[
G_{N,\theta}(\vecx) = \int_{a_1\le\cdots\le a_N}{B_{(a_1,\dots,a_N)}(\vecx;\theta)\mu_N(\d a_1,\dots,\d a_N)}
\]
be the Bessel generating function of $\mu_N$, for all $N\ge 1$.
Since $G_{N,\theta}(0^N)=1$, then $\ln{G_{N,\theta}(\vecx)}$ defines a holomorphic function in a neighborhood of $(0^N)\in\C^N$.
In particular, this means that each $\ln{G_{N,\theta}(\vecx)}$ has a Taylor expansion around $(0^N)$ of the form
\begin{equation}\label{eq:taylor_1}
\ln{G_{N,\theta}(\vecx)} = \sum_{\nu\colon\ell(\nu)\le N}{ b_\nu^N M_\nu(\vecx) },
\end{equation}
where the sum is over partitions $\nu$ of length $\ell(\nu)\le N$, and $M_\nu(\vecx)$ are the monomial symmetric polynomials.
The Taylor coefficients $b_\nu^N$ can be calculated as
\begin{equation}\label{eq:explicit_taylor}
    b_\nu^N = \frac{\partial^{|\nu|}}{\partial x_1^{\nu_1}\cdots \partial x_{\ell(\nu)}^{\nu_{\ell(\nu)}}}{\,\ln{G_{N,\theta}(\vecx)}}\bigg|_{\setzeroes}.
\end{equation}
To state the main theorem, we need to expand $\ln{G_{N,\theta}(\vecx)}$ in the power sum basis $\{p_\la(\vecx)\}_{\la\colon\ell(\la)\le N}$, where
\[
p_\la(\vecx):=\prod_{i=1}^{\ell(\la)}{p_{\la_i}(\vecx)}.
\]
Since one can write $M_{\nu}(\vecx)$ as a unique linear combination of power sums if  and only if $|\nu|\le N$, such expansion is well-defined only up to degree $N$.

\begin{definition}\label{def:appropriate}
    The sequence $\{\mu_N\in\M_N^{>0}\}_{N\ge 1}$ \textbf{is $\theta$-LLN-appropriate} if the corresponding Bessel generating functions $G_{N,\theta}$ are such that, if we expand $\ln{G_{N,\theta}(\vecx)}$ up to degree $N$,
    \begin{equation}\label{eq:taylor_2}
    \ln{G_{N,\theta}(\vecx)} = \sum_{\la\colon|\la|\le N}{ a_\la^N p_\la(\vecx) }+O\big(\|x\|^{N+1}\big),
    \end{equation}
    then there exists a sequence of real numbers $\{\kappa_d\}_{d\ge 1}$ such that
    \begin{enumerate}[label=(\alph*)]
        \item $\displaystyle\lim_{N\to\infty}{\frac{a_{(d)}^N}{N}} = \frac{\theta^{1-d}}{d}\cdot\kappa_d$, for all $d\in\Z_{\ge 1}$,
        \item $\displaystyle\lim_{N\to\infty}{\frac{a_\la^N}{N^{\ell(\la)}}}=0$, whenever $\ell(\la)\ge 2$.
    \end{enumerate}
\end{definition}

\medskip

\begin{definition}
    Let $k\ge 1$ be an arbitrary integer. A \textbf{Łukasiewicz path of length $k$} is a lattice path $\Gamma$ that starts at $(0,0)$, ends at $(k,0)$, has steps of the form $(1,d)$, for some $d\in\{-1,0,1,2,\dots\}$, and stays in the first quadrant, i.e.~does not contain vertices with negative $y$-coordinates. We denote the set of all Łukasiewicz paths of length $k$ by $\mathbf{L}(k)$.
\end{definition}

\begin{theorem}[Main Theorem]\label{thm:main}
    Let $\theta>0$ be fixed.
    The sequence $\{\mu_N\}_{N\ge 1}$ satisfies the fixed temperature LLN (according to Definition~\ref{def:lln}) if and only if it is $\theta$-LLN-appropriate (according to Definition~\ref{def:appropriate}). In that case, the sequences $\{m_k\}_{k\ge 1}$ and $\{\kappa_d\}_{d\ge 1}$ from Definitions \ref{def:lln} and \ref{def:appropriate} are uniquely determined by each other, according to the equalities:
    \begin{equation}\label{eq_fixedtempmoment-cumulant}
        m_k = \sum_{\Gamma\in\Luk(k)}{ \prod_{d\ge 1}{(\kappa_d)^{\#\text{steps $(1,d-1)$ of }\Gamma}} },\text{ for all }k\ge 1.
    \end{equation}
\end{theorem}

The coefficients $a_\la^N$ in the expansion~\eqref{eq:taylor_2} can be expressed in terms of $\ln{G_{N,\theta}(\vecx)}$.
This follows from the explicit expression for the monomial symmetric polynomials $M_\mu(\vecx)$ in terms of the power-sum basis $p_\la(\vecx)$ (see, e.g., \cite[Prop.~2.13]{Z}), as well as equations \eqref{eq:taylor_1}, \eqref{eq:explicit_taylor} and \eqref{eq:taylor_2}.
Explicitly,
\begin{equation}\label{eq:a_coeffs}
a_\la^N = \frac{(-1)^{\ell(\la)-\ell(\nu)}}{\prod_{j\ge 1}{j^{m_j(\la)}m_j(\la)!}} \sum_{\big( \kappa^{(1)},\dots,\kappa^{(\ell(\nu))} \big)} \prod_{i=1}^{\ell(\nu)} \frac{(\ell(\kappa^{(i)})-1)!\,\nu_i}{\prod_{j\ge 1}{m_j(\kappa^{(i)})!}}\cdot
\frac{\partial^{|\nu|}}{\partial x_1^{\nu_1}\cdots x_{\ell(\nu)}^{\nu_{\ell(\nu)}}}\ln{G_{N,\theta}(\vecx)}\bigg|_{\setzeroes},
\end{equation}
where the sum is over all sequences of partitions with $|\kappa^{(i)}|=\nu_i$ and $\kappa^{(1)}\cup\cdots\cup\kappa^{(\ell(\nu))}=\la$.

In this way, conditions (a)-(b) in Definition~\ref{def:appropriate} can be rewritten in terms of the Taylor coefficients of $\ln{G_{N,\theta}(\vecx)}$.
As a result, Theorem~\ref{thm:main} answers the question posed in \cite{BCG} of finding necessary and sufficient conditions for a sequence $\{\mu_N\}_{N\ge 1}$ to satisfy the fixed temperature LLN in terms of the corresponding Bessel generating functions.
Moreover, equation~\eqref{eq:a_coeffs} coincides with the preliminary calculations in that appendix.
For example, for the two partitions of size $2$, we have
\[
a_{(2)}^N = \left\{ \frac{\partial^2}{\partial x_1^2} - \frac{\partial^2}{\partial x_1\partial x_2} \right\}{\ln{G_{N,\theta}(\vecx)}}\bigg|_{\setzeroes}, \qquad
a_{(1,1)}^N = \frac{\partial^2}{\partial x_1\partial x_2}{\ln{G_{N,\theta}(\vecx)}}\bigg|_{\setzeroes},
\]
and for the three partitions of size $3$, we have
\begin{align*}
a_{(3)}^N &= \left\{ \frac{1}{2}\cdot\frac{\partial^3}{\partial x_1^3} - \frac{3}{2}\cdot\frac{\partial^3}{\partial x_1^2\partial x_2} + \frac{\partial^3}{\partial x_1\partial x_2\partial x_3} \right\}{\ln{G_{N,\theta}(\vecx)}}\bigg|_{\setzeroes},\\
a_{(2,1)}^N &= \left\{ \frac{\partial^3}{\partial x_1^2\partial x_2} - \frac{\partial^3}{\partial x_1\partial x_2\partial x_3} \right\}{\ln{G_{N,\theta}(\vecx)}}\bigg|_{\setzeroes},\\
a_{(1,1,1)}^N &= \frac{\partial^3}{\partial x_1\partial x_2\partial x_3} {\ln{G_{N,\theta}(\vecx)}}\bigg|_{\setzeroes}.
\end{align*}

\begin{example}
     Let $a(N)=(a,\dots,a)$, for all $N\ge 1$, where $a$ is a Gaussian random variable of mean $0$ and variance $N$. Since $B_{(a,\dots,a)}(\vecx;\theta)=\exp\left(a\sum_{i=1}^N{x_i}\right)$ (by taking a limit of Eqn.~\eqref{eq_besselbranching}), we have  
    \begin{multline*}
        G_{N,\theta}(\vecx) = \int_{-\infty}^{\infty}{\frac{e^{-\frac{a^{2}}{2N}}}{\sqrt{2\pi N}}\cdot\exp\left(a\sum_{i=1}^{N}x_{i}\right)da}\\
        = \exp\left(\frac{N}{2}p_{1}(\vecx)^{2}\right) = \exp\left(\frac{N}{2}M_{(2)}(\vecx)+NM_{(1,1)}(\vecx)\right).
    \end{multline*}
    One can check that $a(N)$ satisfies the fixed temperature LLN with $m_k=0$ for all $k$. This is revealed by the fact that $\lim_{N\rightarrow\infty}\frac{a^N_{(1,1)}}{N^2}=0$ and $a^N_\la=0$ for all $\la\ne (1,1)$. On the other hand, $G_{N,\theta}(\vecx)$ does not satisfy the conditions proposed by \cite[Claim 9.1]{BCG}, which are given by monomials of partial derivatives.
\end{example}

\section{Sufficient conditions for the LLN}\label{sec:sufficient}

In this section, we assume that the sequence $\{\mu_N\in\M_N^{>0}\}_{N\ge 1}$ is $\theta$-LLN-appropriate, in the sense that, if we let
\begin{equation}\label{eqn:suff_assumption_1}
F_{N,\theta}(\vecx) := \ln{G_{N,\theta}(\vecx)} = \sum_{\la\colon\ell(\la)\le N}{a_\la^N p_\la(\vecx)},
\end{equation}
then there exist $\kappa_1,\kappa_2,\dots\in\R$ such that conditions (a)-(b) from Definition~\ref{def:appropriate} are satisfied, i.e.
\begin{equation}\label{eqn:suff_assumption_2}
\lim_{N\to\infty}{\frac{a_{(d)}^N}{N}} = \frac{\theta^{1-d}\kappa_d}{d},\text{ for all } d\in\Z_{\ge 1},\qquad\quad
\lim_{N\to\infty}{\frac{a_\la^N}{N^{\ell(\la)}}}=0, \text{ if }\ell(\la)\ge 2.
\end{equation}

\medskip

The goal of this section is to prove that the numbers $m_1,m_2,\dots\in\R$, defined by equation~\eqref{eq_fixedtempmoment-cumulant}, satisfy the conditions of Definition~\ref{def:lln}.
By virtue of Proposition~\ref{prop:action_dunkl}, this is equivalent to proving
\begin{equation}\label{eqn:goal_section_4}
\lim_{N\to\infty}{ N^{-|\la|-\ell(\la)}\cdot \left( \prod_{i=1}^{\ell(\la)} \P_{\la_i} \right) e^{F_{N,\theta}(\vecx)} \,\bigg|_{\vecx=(0^N)} } = \prod_{i=1}^{\ell(\la)}{m_{\la_i}},
\end{equation}
for all partitions $\la$, where 
\begin{equation}\label{eq:luk_expression}
m_s := \sum_{\Gamma\in\Luk(s)}{ \prod_{d\ge 1}{(\kappa_d)^{\#\text{steps $(1,d-1)$ of }\Gamma}} },\text{ for all } s\in\Z_{\ge 1}.
\end{equation}

\subsection{Technical lemmas on the action of Dunkl operators}

Define the $k$-truncated polynomial
\[
F^{(k)}_{N,\theta}(\vecx) := \sum_{\la\colon |\la|\le k,\ \ell(\la)\le N}{a_\la^N p_\la(\vecx)}.
\]

\begin{lemma}\label{lem:tech_1}
    Let $1\le j_1,\dots,j_r\le N$ be arbitrary, and let $k=j_1+\dots+j_r$. Then
    \[
    \left[ \prod_{i=1}^r\D_{j_i} \right] e^{F_{N,\theta}(\vecx)} \Big|_{\zeroes} = \left[ \prod_{i=1}^r\D_{j_i} \right] e^{F^{(k)}_{N,\theta}(\vecx) } \Big|_{\zeroes}.
    \]
\end{lemma}
\begin{proof}
    Let $M_\mu(\vecx)$ denote the monomial symmetric polynomials.
    Both sets $\{p_\la(\vecx)\}_{\la\colon\ell(\la)\le N}$ and $\{M_\mu(\vecx)\}_{\mu\colon\ell(\mu)\le N}$ are homogeneous bases of the ring of symmetric polynomials with real coefficients in the variables $x_1,\dots,x_N$, so we can write $F_{N,\theta}(\vecx)$ in terms of $M_\mu(\vecx)$.
    Then the truncation $F_{N,\theta}^{(k)}(\vecx)$ is the result of ignoring the $M_\mu(\vecx)$ with $|\mu|>k$.
    With this switch to monomial symmetric polynomials, the lemma follows exactly from the same argument as in \cite[Lemma~5.2]{BCG}.
\end{proof}

\begin{lemma}\label{lem:tech_2}
    If for all partitions $\mu$, we have the limit
    \begin{equation}\label{eq:assumption_tech_2}
    \lim_{N\to\infty}{ N^{-|\mu|}\left[ \prod_{i=1}^{\ell(\mu)}{\D_i^{\mu_i}} \right] e^{F_{N,\theta}(\vecx)} \bigg|_{\vecx=(0^N)} } = \prod_{i=1}^{\ell(\mu)}{m_{\mu_i}},
    \end{equation}
    then for all partitions $\la$, we have
    \begin{equation}\label{eq:conclusion_tech_2}
    \lim_{N\to\infty} N^{-|\la|-\ell(\la)}\left[ \prod_{i=1}^{\ell(\la)}{\P_{\la_i}} \right] e^{F_{N,\theta}(\vecx)} \bigg|_{\vecx=(0^N)} = \prod_{i=1}^{\ell(\la)}{m_{\la_i}}.
    \end{equation}
\end{lemma}
\begin{proof}
Assume that Eqn.~\eqref{eq:assumption_tech_2} holds for all $\mu$. Let us take any partition $\la$ and prove Eqn.~\eqref{eq:conclusion_tech_2}.

If $\ell(\la)=1$, then $\la=(n)$, for some $n\in\Z_{\ge 1}$.
Note that $\D_i^n e^{F_{N,\theta}(\vecx)}$ differs from $\D_j^n e^{F_{N,\theta}(\vecx)}$ only by the change of variables $x_i\leftrightarrow x_j$ because of the symmetry of $F_{N,\theta}(\vecx)$; in particular,
\begin{equation}\label{eq:equality_1}
\D_i^n e^{F_{N,\theta}(\vecx)}\big|_{\vecx=(0^N)} = \D_j^n e^{F_{N,\theta}(\vecx)}\big|_{\vecx=(0^N)}.
\end{equation}
As a result,
\begin{equation*}
N^{-n-1}\,\P_n e^{F_{N,\theta}(\vecx)} \Big|_{\vecx=(0^N)} = N^{-n-1}\sum_{i=1}^n{\D_i^n} e^{F_{N,\theta}(\vecx)} \Big|_{\vecx=(0^N)} = N^{-n}\,\D_1^n e^{F_{N,\theta}(\vecx)}\Big|_{\vecx=(0^N)},
\end{equation*}
proving that the desired Eqn.~\eqref{eq:conclusion_tech_2} follows at once from Eqn.~\eqref{eq:assumption_tech_2} for $\mu=(n)$.

\smallskip

If $\ell(\la)=2$, then $\la=(n_1,n_2)$, for some positive integers $n_1\ge n_2$. As above, we deduce that $\D_i^{n_2}\D_j^{n_1} e^{F_{N,\theta}(\vecx)}\big|_{\vecx=(0^N)} = \D_2^{n_2}\D_1^{n_1} e^{F_{N,\theta}(\vecx)}\big|_{\vecx=(0^N)}$, whenever $i\ne j$.
Together with~\eqref{eq:equality_1}, we obtain
\begin{multline*}
N^{-n_1-n_2-2}\,\P_{n_2}\P_{n_1} e^{F_{N,\theta}(\vecx)} = N^{-n_1-n_2-2}\sum_{i=1}^N{\D_i^{n_2}}\sum_{j=1}^N{\D_j^{n_1}} e^{F_{N,\theta}(\vecx)}\\
= N^{-n_1-n_2-2}\cdot N(N-1)\D_2^{n_2}\D_1^{n_1} e^{F_{N,\theta}(\vecx)} + N^{-n_1-n_2-2}\cdot N\,\D_1^{n_1+n_2} e^{F_{N,\theta}(\vecx)}\\
= \left(1-\frac{1}{N}\right)\cdot N^{-n_1-n_2}\,\D_2^{n_2}\D_1^{n_1}e^{F_{N,\theta}(\vecx)} + \frac{1}{N}\cdot N^{-n_1-n_2}\,\D_1^{n_1+n_2}e^{F_{N,\theta}(\vecx)}.
\end{multline*}
By assumption~\eqref{eq:assumption_tech_2}, the first term above converges to $m_1m_2$ and the second one to $0$, as $N\to\infty$.
Then Eqn.~\eqref{eq:conclusion_tech_2} follows.

\smallskip

If $\ell(\la)=\ell>2$, the proof is similar.
Indeed,
\begin{equation}\label{eq:equality_2}
\prod_{i=1}^{\ell(\la)}{\P_{\la_i}} = \prod_{i=1}^{\ell(\la)}{\sum_{j=1}^N{\D_j^{\la_i}}} = \sum_{\substack{1\le j_1,\dots,j_{\ell(\la)}\le N \\ j_1,\dots,j_{\ell(\la)}\text{ are distinct}}}{\D_{j_1}^{\la_1}\cdots\D_{j_{\ell(\la)}}^{\la_{\ell(\la)}}} + p(\D_1,\dots,\D_N),
\end{equation}
where $p(x_1,\dots,x_N)$ is a symmetric polynomial, where each monomial has at most $\ell(\la)-1$ distinct variables $x_j$.
Again, we have that if $j_1,\dots,j_{\ell(\la)}$ are distinct, then $\D_{j_1}^{\la_1}\cdots\D_{j_{\ell(\la)}}^{\la_{\ell(\la)}} e^{F_{N,\theta}(\vecx)}\big|_{\vecx=(0^N)}$ is equal to $\D_1^{\la_1}\cdots\D_{\ell(\la)}^{\la_{\ell(\la)}} e^{F_{N,\theta}(\vecx)}\big|_{\vecx=(0^N)}$. This shows, as in the case when $\ell(\la)=2$, that~\eqref{eq:equality_2} gives
\begin{multline*}
N^{-|\la|-\ell(\la)} \left[\prod_{i=1}^{\ell(\la)}{\P_{\la_i}}\right]  e^{F_{N,\theta}(\vecx)} \bigg|_{\vecx=(0^N)} = \left( 1 + O\Big( \frac{1}{N} \Big) \right)\cdot N^{-|\la|} \D_1^{\la_1}\cdots\D_{\ell(\la)}^{\la_{\ell(\la)}} e^{F_{N,\theta}(\vecx)}\big|_{\zeroes}\\
+ \sum_{j=1}^{\ell(\la)-1} \sum_{\substack{k_1,\dots,k_j\in\Z_{\ge 1} \\ k_1+\dots+k_j=|\la|}} O\Big( \frac{1}{N} \Big)\cdot N^{-k_1-\,\dots\,-k_j} \D_1^{k_1}\cdots\D_j^{k_j} e^{F_{N,\theta}(\vecx)}\big|_{\zeroes},
\end{multline*}
and then assumption~\eqref{eq:assumption_tech_2} gives the desired equation~\eqref{eq:conclusion_tech_2}.
\end{proof}

\subsection{Proof of the ``if'' part of the main theorem}

By Lemma~\ref{lem:tech_2} and the commutativity of Dunkl operators, the desired equation~\eqref{eqn:goal_section_4} will follow if, for any $s\in\Z_{\ge 1}$ and $\mu_1,\dots,\mu_s\in\Z_{\ge 1}$, we manage to prove
\begin{equation}\label{eqn:goal2_section_4}
\lim_{N\to\infty}{ N^{-|\mu|}\Big[ \D_s^{\mu_s}\cdots\D_1^{\mu_1} \Big] e^{ F_{N,\theta}(\vecx) } \,\Big|_{\vecx=(0^N)} } = \prod_{i=1}^s{ m_{\mu_i} }.
\end{equation}
Moreover, due to Lemma~\ref{lem:tech_1}, we can and will assume that $F_{N,\theta}(\vecx)$ is a symmetric polynomial of degree $\mu_1+\dots+\mu_s$.
Proving the limit~\eqref{eqn:goal2_section_4} will occupy the rest of this section.

\medskip

\textbf{Step 1.} Let $\mu:=(\mu_1,\dots,\mu_s)$ and denote $|\mu|:=\mu_1+\dots+\mu_s$.
Note that $\D_s^{\mu_s}\cdots\D_1^{\mu_1}$ is a sum of $N^{|\mu|}$ terms $w^s_{\mu_s}\cdots w^s_1\cdots w^1_{\mu_1}\cdots w^1_1$, where each $w^a_b$ is of the form $\partial_a$, or $\frac{\theta}{x_a-x_j}(1-\sigma_{a,j})$, for some $j\ne a$.
Before taking the limit, we need to compute $\Big[ \D_s^{\mu_s}\cdots\D_1^{\mu_1} \Big] e^{ F_{N,\theta}(\vecx) }$, which is the sum of all terms $\big(w^s_{\mu_s}\cdots w^s_1\cdots w^1_{\mu_1}\cdots w^1_1\big)e^{F_{N,\theta}(\vecx)}$.
If $w^1_1$ is of the form $\frac{\theta}{x_1-x_j}(1-\sigma_{1,j})$, then it kills $e^{ F_{N,\theta}(\vecx) }$, for being symmetric in all variables $x_i$, therefore we can ignore those terms and keep only the terms with $w^1_1=\partial_1$.
Evidently, $\partial_1e^{ F_{N,\theta}(\vecx) } = \big( \partial_1 F_{N,\theta}(\vecx) \big) e^{ F_{N,\theta}(\vecx) }$.
After that, we apply the operators $w^1_2,w^1_3,\dots,w^s_{\mu_s}$ (in that order) to expressions of the form $h(\vecx)e^{F_{N,\theta}(\vecx)}$, for some polynomial $h(\vecx)$, and keep obtaining expressions of this form, as observed by the equalities:
\begin{gather}
\partial_i\left( h(\vecx)\, e^{F_{N,\theta}(\vecx)} \right) = \partial_i h(\vecx) e^{F_{N,\theta}(\vecx)} + \left(\sum_{\nu\colon|\nu|\le|\mu|}\ \sum_{j=1}^{\ell(\nu)}{ \nu_j a_\nu^N x_i^{\nu_j-1}\prod_{k\ne j}{p_{\nu_k}(\vecx)}} \right) h(\vecx)\,e^{F_{N,\theta}(\vecx)},\label{eq:partial_i}\\
\nonumber\\
\theta\cdot\frac{1-\sigma_{a,j}}{x_a-x_j}\left( h(\vecx)\,e^{F_{N,\theta}(\vecx)} \right) = \ \theta\cdot\frac{1-\sigma_{a,j}}{x_a-x_j}\big( h(\vecx) \big)\,e^{F_{N,\theta}(\vecx)}.\nonumber
\end{gather}
Then, to obtain the limit of $\Big[ \D_s^{\mu_s}\cdots\D_1^{\mu_1} \Big] e^{ F_{N,\theta}(\vecx) }$, we proceed as follows.

First, expand each term $\big(w^s_{\mu_s}\cdots w^s_1\cdots w^1_{\mu_1}\cdots w^1_1\big)e^{F_{N,\theta}(\vecx)}$ (by the above rules) as a sum that can be finally be expressed as $H(\vecx)e^{F_{N,\theta}(\vecx)}$, for some polynomial $H(\vecx)$. 
This will also be a polynomial in $N,\theta$, as well as $\{a_{(d)}^N : d\in\Z_{\ge 1}\}$ and $\{a_\la^N : \ell(\la)\ge 2\}$.

Second, we set $\vecx\mapsto(0^N)$. Since $F_{N,\theta}(\vecx)\big|_{\vecx=(0^N)}=0$, then $H(\vecx)e^{F_{N,\theta}(\vecx)}\big|_{\vecx=(0^N)}=H(0^N)$ is the constant coefficient of $H(\vecx)$, which is a polynomial in $N$, $\theta$, $\{a_{(d)}^N : d\in\Z_{\ge 1}\}$ and $\{a_\la^N : \ell(\la)\ge 2\}$.

Finally, we multiply by $N^{-|\mu|}$ and take the limit of the resulting expression as $N\to\infty$, using assumptions~\eqref{eqn:suff_assumption_2}. The result will be a polynomial in $\theta$ and $\{\kappa_d : d\in\Z_{\ge 1}\}$. 
We need to prove that the final result is the right hand side of \eqref{eqn:goal2_section_4}.

\smallskip

\textbf{Step 2.} Let $\nu^{(1)},\dots,\nu^{(k)}$ be partitions such that constant multiples of the product $a_{\nu^{(1)}}^N\cdots a_{\nu^{(k)}}^N$ appear in terms coming from some $\big(w^s_{\mu_s}\cdots w^s_1\cdots w^1_{\mu_1}\cdots w^1_1\big)e^{F_{N,\theta}(\vecx)}$.
We want to calculate the contribution of those terms after setting $\vecx\mapsto(0^N)$, multiplying by $N^{-|\mu|}$, and taking the limit $N\to\infty$. By the rule \eqref{eq:partial_i}, we see that $a_{\nu^{(r)}}^N$ appears for the first time after an application of some $\partial_i$, as part of the expression
\begin{equation}\label{eq:multiple_a}
\nu^{(r)}_j a_{\nu^{(r)}}^N x_i^{\nu^{(r)}_j-1}\prod_{k\ne j}{p_{\nu^{(r)}_k}(\vecx)}\cdot e^{F_{N,\theta}(\vecx)}.
\end{equation}
Since we want the product $a_{\nu^{(1)}}^N\cdots a_{\nu^{(k)}}^N$, we need $k$ partial derivatives among the $w^a_b$ in order to get the $k$ factors $a^N_{\nu^{(r)}}$, $1\le r\le k$.
Then we have another $|\mu|-k$ operators of the form $\partial_a$ or $\frac{\theta}{x_a-x_j}(1-\sigma_{a,j})$ that must act on expressions $h(\vecx)e^{F_{N,\theta}(\vecx)}$ (but $\partial_a$ would act on $h(\vecx)$ and not on $e^{F_{N,\theta}(\vecx)}$) and each of them decreases the degree of $h(\vecx)$ by $1$.
Then we set $\vecx=(0^N)$, so we want the expressions $h(\vecx)e^{F_{N,\theta}(\vecx)}$ with $h(\vecx)$ being a constant at the end.

We note that if $g(\vecx)$ is a symmetric function in the variables $x_i$, then $\frac{\theta}{x_a-x_j}(1-\sigma_{a,j})\big( g(\vecx)h(\vecx) \big) = g(\vecx)\cdot\frac{\theta}{x_a-x_j}(1-\sigma_{a,j})\big( h(\vecx) \big)$. This implies that if $\frac{\theta}{x_a-x_j}(1-\sigma_{a,j})$ is applied to \eqref{eq:multiple_a}, then each of the factors $p_{\nu_k^{(r)}}(\vecx)$, where $1\le k\le\ell(\nu^{(r)})$ and $k\ne j$, is preserved.
But, in the end, these factors cannot be present, or otherwise setting $\vecx\mapsto(0^N)$ would make its contribution vanishing. Thus, we need $\ell(\nu^{(r)})-1$ partial derivatives to act on \eqref{eq:multiple_a}.
Hence, for $a_{\nu^{(1)}}^N\cdots a_{\nu^{(k)}}^N$ to appear inside some $\big(w^s_{\mu_s}\cdots w^s_1\cdots w^1_{\mu_1}\cdots w^1_1\big)e^{F_{N,\theta}(\vecx)}$, one must have at least $k+\sum_{r=1}^k{(\ell(\nu^{(r)})-1)} = \sum_{r=1}^k{\ell(\nu^{(r)})}$ partial derivatives among the $w^a_b$, while the other $|\mu|-\sum_{r=1}^k{\ell(\nu^{(r)})}$ operators can be any of the $N$ operators $\partial_a$ or $\frac{\theta}{x_a-x_j}(1-\sigma_{a,j})$, for some $1\le j\le N$, $j\ne a$.

As a result of the previous analysis, the coefficient of $a_{\nu^{(1)}}^N\cdots a_{\nu^{(k)}}^N$ in $\Big[ \D_s^{\mu_s}\cdots\D_1^{\mu_1} \Big] e^{ F_{N,\theta}(\vecx) } \,\big|_{\vecx=(0^N)}$ is of order $O\big(N^{|\mu|-\sum_{r=1}^k{\ell(\nu^{(r)})}}\big)$, so the coefficient of $a_{\nu^{(1)}}^N\cdots a_{\nu^{(k)}}^N$ in $N^{-|\mu|}\, \D_s^{\mu_s}\cdots\D_1^{\mu_1} e^{ F_{N,\theta}(\vecx) } \,\big|_{\vecx=(0^N)}$ is of order $O\big(N^{-\sum_{r=1}^k{\ell(\nu^{(r)})}}\big)$.
By the assumptions~\eqref{eqn:suff_assumption_2}, the limit $\lim_{N\to\infty}{N^{-\ell(\nu^{(r)})}a^N_{\nu^{(r)}}}$ is finite for all partitions $\nu^{(r)}$, and it is zero when $\ell(\nu^{(r)})\ge 2$.
The conclusion is that the limit in the left hand side of \eqref{eqn:goal2_section_4} exists and is finite and that we can ignore terms that contain some $a^N_\la$, with $\ell(\la)\ge 2$, in the prelimit expression $\Big[ \D_s^{\mu_s}\cdots\D_1^{\mu_1} \Big] e^{ F_{N,\theta}(\vecx) }$, since their final contribution is zero.

\smallskip

\textbf{Step 3.} In Step 2, we showed that the prelimit expression in the left hand size of \eqref{eqn:goal2_section_4} is a polynomial in the variables $\{a_{(d)}^N : d\in\Z_{\ge 1}\}$, $\{a_\la^N : \ell(\la)\ge 2\}$, but the terms that contain some $a_\la^N$, with $\ell(\la)\ge 2$, vanish after taking the limit.
As a result, we can assume w.l.o.g. that $a_\la^N=0$, whenever $\ell(\la)\ge 2$, from the beginning. So, instead of $F_{N,\theta}(\vecx)$ in \eqref{eqn:suff_assumption_1}, we can consider
\[
H_{N,\theta}(\vecx):= \sum_{d\ge 1}{a_{(d)}^N p_{(d)}(\vecx)} = \sum_{d\ge 1}\sum_{i=1}^N{a_{(d)}^N x_i^d},
\]
and finish by proving that
\begin{equation}\label{eqn:goal3_section_4}
\lim_{N\to\infty}{ N^{-|\mu|}\Big[ \D_s^{\mu_s}\cdots\D_1^{\mu_1} \Big] e^{ H_{N,\theta}(\vecx) } \,\Big|_{\vecx=(0^N)} } = \prod_{i=1}^s{ m_{\mu_i} }.
\end{equation}
At this point, the proof of \eqref{eqn:goal3_section_4} follows exactly from the same argument as the proof of \cite[Proposition~5.6]{BCG} (which studied the high temperature regime $N\to\infty$, $N\theta\to\text{const}$) with only minor modifications.
The conclusion from this argument is that
\begin{equation}\label{eq:a_la_bcg}
N^{-|\mu|}\Big[ \D_s^{\mu_s}\cdots\D_1^{\mu_1} \Big] e^{ H_{N,\theta}(\vecx) } \,\Big|_{\vecx=(0^N)} = \prod_{i=1}^s{ \Big( [x^0]\big(\theta\cdot Z + *_{h_N}\big)^{\mu_i-1}\big(h_N(x)\big) \Big) } + O\bigg( \frac{1}{N} \bigg),
\end{equation}
where $\displaystyle h_N(x):=\frac{1}{N}\cdot\partial_1 H_{N,\theta}(\vecx)|_{x_1=x}=\sum_{d=1}^\infty{\frac{d a^N_{(d)}}{N}x^{d-1}}$, the operator $Z$ acts on the ring of formal power series in $x$ by $Z(x^m) := \mathbf{1}_{\{m=0\}}\cdot x^{m-1}$, and $*_{h_N}$ is multiplication by $h_N(x)$.
From \eqref{eq:a_la_bcg} and the assumption \eqref{eqn:suff_assumption_2}, we have
\begin{equation}\label{eq:a_la_bcg_2}
\lim_{N\to\infty}{ N^{-|\mu|}\Big[ \D_s^{\mu_s}\cdots\D_1^{\mu_1} \Big] e^{ H_{N,\theta}(\vecx) } \,\Big|_{\vecx=(0^N)} } = \prod_{i=1}^s{ \Big( [x^0]\big(\theta\cdot Z + *_h\big)^{\mu_i-1}\big(h(x)\big) \Big) },
\end{equation}
where $h(x):=\sum_{d=1}^\infty{\theta^{1-d}\kappa_d x^{d-1}}$.

\smallskip

\textbf{Step 4.} In view of \eqref{eq:a_la_bcg_2}, our desired equation \eqref{eqn:goal3_section_4} would follow from the equalities
\begin{equation}\label{eq:limit_fixed_temp}
[x^0]\big(\theta\cdot Z + *_h\big)^{\ell-1}\big(h(x)\big) = m_\ell,\text{ for all }\ell\in\Z_{\ge 1},
\end{equation}
where $h(x):=\sum_{d=1}^\infty{\theta^{1-d}\kappa_d x^{d-1}}$ and $m_\ell$ is the function of $\kappa_1,\dots,\kappa_\ell$ given by~\eqref{eq:luk_expression}.
This equality can be established by known arguments in the literature, e.g.~see the proof of \cite[Theorem~3.6]{Cue0}.

The argument is the following.
Fix any $\ell\in\Z_{\ge 1}$, and express the left hand side of \eqref{eq:limit_fixed_temp} as the sum
\begin{equation}\label{eq:step4}
[x^0]\big(\theta\cdot Z + *_h\big)^{\ell-1}\big(h(x)\big) = \sum_{w_1,w_2,\dots,w_\ell}{w_\ell\cdots w_2 w_1(1)}
\end{equation}
over sequences $w_1,w_2,\dots,w_\ell$ such that:

$\bullet$ each $w_i$ is the operator of multiplication by $\kappa_d x^{d-1}$, for some $d\in\Z_{\ge 1}$, or the operator $\theta\cdot Z$;

$\bullet$ $w_1=\kappa_{d_1}x^{d_1-1}$, for some $d_1\in\Z_{\ge 1}$;

$\bullet$ $w_\ell\cdots w_2 w_1(1)$ is a constant.

\noindent For any sequence $w_1,w_2,\dots,w_\ell$ satisfying the conditions above, there exist $c_1,\dots,c_\ell\in\R$ and $n_1,\dots,n_\ell\in\Z_{\ge 0}$ such that $n_\ell=0$ and $w_i\cdots w_1(1)=c_i\cdots c_1 x^{n_i}$, for all $i=1,\dots,\ell$.
In fact, if $w_1=\kappa_{d_1}x^{d_1-1}$, then $c_1=\kappa_{d_1}$, $n_1=d_1-1$, and inductively, if $w_i=\kappa_dx^{d-1}$, then $c_i=\kappa_d$, $n_i=n_{i-1}+d-1$, while if $w_i=\theta\cdot Z$, then $c_i=\theta$, $n_i=n_{i-1}-1$.
It can be easily checked that the sequence $(0,0)\to (1,n_1)\to\cdots\to(\ell,n_\ell)=(\ell,0)$ is a Łukasiewicz path of length $\ell$.
Moreover, if we assign weights $c_i$ to the edges $(i-1,n_{i-1})\to(i,n_i)$, then the product of edge-weights is equal to $w_\ell\cdots w_1(1)$, therefore, the right hand side of \eqref{eq:step4} can be written as a sum of weighted Łukasiewicz paths, where the steps $(i,n_i)-(i-1,n_{i-1})$ of the form $(1,d-1)$, where $d\in\Z_{\ge 1}$, have weight $\theta^{1-d}\kappa_d$, and those of the form $(1,-1)$ have weight $\theta$.
As a result,
\begin{equation}\label{eq:luk_expression_2}
\text{right hand side of }\eqref{eq:step4} := \sum_{\Gamma\in\Luk(\ell)}{ \theta^{\,\#\text{steps $(1,-1)$ of }\Gamma}\cdot\prod_{d\ge 1}{(\theta^{1-d}\kappa_d)^{\#\text{steps $(1,d-1)$ of }\Gamma}} }.
\end{equation}
Finally, note that for any Łukasiewicz path $\Gamma$, we have
\[
\#\text{steps $(1,-1)$ of }\Gamma \,=\, \sum_{d\ge 1}{ (d-1)\cdot\#\text{steps $(1,d-1)$ of }\Gamma}.
\]
Indeed, this equality simply says that the downward $y$-displacement of $\Gamma$ is equal to the upward $y$-displacement of $\Gamma$, but this is true because $\Gamma$ begins at $(0,0)$ and ends at $(\ell,0)$, both points belonging to the $x$-axis.
This equality implies that the right hand sides of \eqref{eq:luk_expression_2} and \eqref{eq:luk_expression} coincide.
Hence, together with \eqref{eq:step4}-\eqref{eq:luk_expression_2}, we deduce the desired equation~\eqref{eq:limit_fixed_temp}, thus \textbf{concluding the proof of the ``if part'' of the main theorem}.

\section{Necessary conditions for the LLN}\label{sec:necessary}

In this section, we assume that $\{\mu_N\}_{N\ge 1}$ satisfies the fixed-temperature LLN and prove that the conditions of Definition~\ref{def:appropriate} (namely, $\theta$-LLN-appropriateness) are satisfied.

\subsection{Preliminaries on constellations}

All our closed surfaces will be connected.
Closed surfaces are classified by their genus, which satisfies
\[
2-2g = V-E+F,
\]
where $V$, $E$ and $F$ denote the number of vertices, edges and faces of any embedded graph on the surface; moreover, $g$ is the genus of the surface.
Up to homeomorphism, the genera of orientable closed surfaces are nonnegative integers, while the genera of non-orientable surfaces are nonnegative half-integers.

\begin{definition}\label{def:map}
    A  \textbf{map} $\MM$ is an embedding of a graph, possibly with multiple edges and loops, on a closed (orientable or non-orientable) surface, such that each face of the graph is simply connected. 
\end{definition}

\begin{definition}\label{def:const}
Let $k\ge 1$ be an integer. A $\mathbf{k}$-\textbf{constellation} $\MM$ is a map, equipped with a coloring of its vertices with colors from the set $\{0, 1,\dots, k\}$, such that the following conditions are satisfied.

\smallskip

    $\bullet$ Each vertex of color 0 is only connected to vertices of color 1.

\smallskip

    $\bullet$ Each vertex of color $k$ is only connected to vertices of color $k-1$.

\smallskip

    $\bullet$ For $1\le i\le k-1$, each corner at a vertex of color $i$ separates vertices of colors $i-1$ and $i+1$.

\smallskip

\noindent Finally, we will color all corners, too, with the same colors as the ones of the corresponding vertices.
\end{definition}

The \textbf{size} of $\MM$, denoted by $|\MM|$, is the number of corners of color $0$ in $\MM$.
Note that the number of edges of any $k$-constellation $\MM$ is $k\cdot|\MM|$.

We say that $\MM$ is \textbf{connected} if the underlying graph is connected, and $\MM$ is \textbf{orientable/non-orientable} if its underlying surface is orientable/non-orientable. The genus of $\MM$, denoted $g(\MM)$, is the genus of the underlying surface.

An \emph{oriented corner} of $\MM$ is a pair $(c,\epsilon)$, where $c$ is a corner and $\epsilon\in\{-1,+1\}$.
We say that $\MM$ is \emph{rooted} if it has a distinguished oriented corner of color $0$, which is called the \emph{root of $\MM$}.

By definition, in any $k$-constellation $\MM$, each vertex of color $k$ is connected only to vertices of color $k-1$.
Thus, we can add a new vertex of color $k+1$ for each corner of color $k$ and join it with an edge to the corresponding vertex of color $k$ to obtain a $(k+1)$-constellation of the same size.
This size-preserving map from $k$-constellations to $(k+1)$-constellations preserves rooted, connected constellations. Let $C_{d,\bullet}^{(k)} $ be the set of rooted, connected $k$-constellations of size $d$; the previous discussion describes inclusion maps
\[
i_k\colon C_{d,\bullet}^{(k)} \hookrightarrow C_{d,\bullet}^{(k+1)},\quad k\ge 1.
\]
A rooted, connected, infinite constellation of size $d$ is any sequence $(\MM^{(1)}, \MM^{(2)}, \cdots)$ such that $\MM^{(k)}\in C_{d,\bullet}^{(k)}$ and $i_k\big( \MM^{(k)} \big) = \MM^{(k+1)}$, for all $k\ge 1$.
It is convenient to think of an infinite constellation as the embedding of an infinite graph on a surface, obtained from the construction above. We denote the set of rooted, connected, infinite constellations by $C_{d,\bullet}^{(\infty)}$.

\begin{definition}\label{def:dataconstellation}
For $\MM\in C_{d,\bullet}^{(\infty)}$, let $F(\MM), V(\MM)$ be the sets of faces and vertices of $\MM$. Moreover:

    $\bullet$ For any $f\in F(\MM)$, denote by $\deg(f)$ the number of corners of color $0$ in $f$. Also, denote by $\pmb{\mu_1}(\MM)$ the partition with parts $\deg(f)$, $f\in F(\MM)$.

\smallskip

    $\bullet$ Denote by $V_0(\MM)$ the set of vertices of $\MM$ of color $0$, and $v_0(\MM):=|V_0(\MM)|$.

\smallskip

    $\bullet$ For any $v\in V(\MM)$, denote by $\deg(v)$ the number of edges adjacent to $v$.\footnote{Note that $\deg(v)$ is not necessarily the number of vertices adjacent to $v$, because multiple edges are allowed.} Also, denote by $\pmb{\mu_2}(\MM)$ the partition with parts $\deg(v)$, $v\in V_0(\MM)$.

\smallskip
    
    $\bullet$ For any $i\ge 1$, denote by $v_i(\MM)$ the number of vertices of $\MM$ of color $i$.
    Consider the sequence $$\eta(\MM):=\big(|\MM|-v_i(\MM)\big)_{i\ge 1}.$$
We say that $\MM$ is \textbf{normal} if $v_1(\MM)\le v_2(\MM)\le\cdots$.
Since $v_i(\MM)\le|\MM|$, for all $i$, normality is equivalent to the fact that $\eta(\MM)$ is a partition.
\end{definition}

\noindent By the definitions, note that $|\pmb{\mu_{1}}(\MM)|=|\pmb{\mu_{2}}(\MM)|=|\MM|$.

\begin{proposition}\cite[Eqn. (5)]{CD}\label{prop:euler}
    For any $\MM\in C_{d,\bullet}^{(\infty)}$ of genus $g$, we have
\begin{equation}
    |\eta(\MM)|-\ell\big(\pmb{\mu_1}(\MM)\big) -\ell\big(\pmb{\mu_2}(\MM)\big) = 2g-2.
\end{equation}
\end{proposition}

\begin{proof}
    Take a large enough $k$, so that $\eta_i(\MM)=0$, for all $i\ge k$. Then we can regard $\MM$ as a k-constellation. It follows that the number of edges of $\MM$ is $k\cdot |\MM|$, the number of vertices of $\MM$ is $\sum_{i=0}^k{v_i(\MM)} = \sum_{i=1}^k{v_i(\MM)}+\ell(\pmb{\mu_2}(\MM))$, and the number of faces of $\MM$ is $\ell(\pmb{\mu_1}(\MM))$. Since $E-V-F=2g-2$, the result follows.
\end{proof}

\subsection{An analytic expansion of the multivariate Bessel function}\label{sec:CD_bessel}

Consider the sequences of variables $\bfp=(p_1,p_2,\dots)$, $\bfq=(q_1,q_2,\dots)$, and $(g_1,g_2,\dots)$.
Let $J_\la(\bfp;\theta)$ be the Jack symmetric function defined by \cite[Ch.~VI.10, Eqn.~(10.13)]{M}, where $\bfp$ is viewed as the sequence of power sum symmetric functions.
Then $J_\la(\bfp;\theta)$ is a polynomial of degree $|\la|$, if each $p_k$ is regarded as a variable of degree $k$.
Define $J_\la(\bfq;\theta)$ analogously.
Further, consider the generating series $G(z)=1+\sum_{k\ge 1}g_{k}z^{k}$.
Following~\cite{CD}, define the formal power series
\begin{multline}\label{eq_Fhat}
\hat{F}^G(t,\bfp,\bfq,g_1,g_2,\dots;\theta) := \sum_{d\ge 0} t^d\sum_{|\la|=d} J_\la(\bfp;\theta)J_\la(\bfq;\theta)\\
\prod_{(i,j)\in\la}\frac{\big((\la_i-j)+\theta(\la_j'-i)+\theta\big)}{\big((\la_i-j)+\theta(\la_j'-i)+1\big)} \prod_{(i,j)\in \la}G(c_\theta(i,j)),
\end{multline}
where $\la'_j$ denotes the number of boxes of the $j$-th column of $\la$ and $c_\theta(i,j) := \theta^{-1}(j-1)-(i-1)$.

\begin{theorem}\cite[Theorem 6.2]{CD}\label{thm:CDmaintheorem}
There exists a map\footnote{The map $\nu$ measures the ``non-orientability'' of a connected rooted constellation $\MM$: see \cite[Section 3]{CD}.} \[
    \nu\colon\bigcup_{d\ge 1} C_{d,\bullet}^{(\infty)}\to\Z_{\ge 0},
    \]
    such that $\nu(\MM)=0$ if and only if $\MM$ is orientable (\cite[Definition-Lemma 3.8]{CD}), and such that we have the following equality of formal power series:
    \begin{equation}
        \ln\Big(\hat{F}^{G}\Big) =
        \theta\cdot\sum_{d\ge 1} \frac{t^d}{d} \sum_{\MM\in C^{(\infty)}_{d,\bullet}: \ \MM \text{ normal}}\big(\theta^{-1}-1\big)^{\nu(\MM)}
        p_{\pmb{\mu_1}(\MM)}
        q_{\pmb{\mu_2}(\MM)}\cdot
        f_{\eta(\MM)}\left(g_1,g_2,\dots\right),
    \end{equation}
    where $f_{\eta}(e_1,e_2,\dots)$ is the polynomial that expresses the monomial symmetric function $m_{\eta}$ in terms of the elementary symmetric polynomials $e_1,e_2,\dots$, and $\MM$ in the inner sum ranges over all rooted, connected, normal, infinite constellations of size $d$.
\end{theorem}

While Theorem~\ref{thm:CDmaintheorem} is an identity of formal power series, we can turn it into an analytic equality for the logarithm of the multivariate Bessel function as follows.

\begin{theorem}\label{thm:keyexpansion}
    Fix $n\in\N$, $\theta>0$, and set $\alpha=\theta^{-1}$. Let $N\in\N$ be large enough so that $N>\max(\alpha,1)\cdot n$. Then, as an analytic function of $(a,\vecx)\in \C^N\times \C^{N}$, we have the expansion
    \begin{multline}\label{eq_keyexpansion}
        \ln\Big(B_a(\vecx;\theta)\Big) = \sum_{d=1}^n\frac{\alpha^{d-1}}{d} \sum_{g\in\frac{1}{2}\Z_{\ge 0}} N^{2-2g} \sum_{\substack{\MM\in C^{(\infty)}_{d,\bullet} \\ \MM\text{ normal},\ g(\MM)=g}}(\alpha-1)^{\nu(\MM)} \prod_{f\in F(\MM)}\frac{p_{deg(f)}(\vecx)}{N}\\
        \times\prod_{v\in V_0(\MM)}\frac{p_{deg(v)}(a/N)}{N}\cdot f_{\eta(\MM)}\Big(\big\{(-1)^k\big\}_{k\ge 1}\Big)+O\big(\|x\|^{n+1}\big),
    \end{multline}
where $f_{\eta}$ and $\MM$ are defined in the same way as in Theorem~\ref{thm:CDmaintheorem}.
The sum over $g\in\frac{1}{2}\Z_{\ge 0}$ in the right hand side converges absolutely.
\end{theorem}
\begin{proof}
The multivariate Bessel functions $B_a(\vecx;\theta)$ are analytic functions of $(a,\vecx)\in\C^N\times\C^N$ and can be expanded in terms of Jack polynomials as follows (see~\cite[Sec.~4]{Ok-Ol}):
\begin{equation}\label{eq_BesselJack}
    B_a(\vecx;\theta) = \sum_{\ell(\la)\le N}\frac{\prod_{(i,j)\in\la}\big((\la_i-j)+\theta(\la_j'-i)+\theta\big)}{\prod_{(i,j)\in\la}\big((\la_i-j)+\theta(\la_j'-i) + 1\big) \big(N\theta+j-1-\theta(i-1)\big)} \,J_\la(\vecx;\theta)J_\la(a;\theta).
\end{equation}

On the other hand, consider the power series $\hat{F}^G$ obtained from specializing to the sequence $g_k=(-N)^{-k}$, or equivalently, from $G(z)=(1+\frac{z}{N})^{-1}$, and setting $t=(N\theta)^{-1}$. Furthermore, for  $\vecx=(x_1,\dots,x_N)\in \C^N$, $a=(a_1,\dots,a_N)\in \C^N$, we specify $p_k$ and $q_k$ as $p_k(\vecx)=\sum_{i=1}^N{x_i^k}$, $q_k(a)=\sum_{i=1}^N{a_i^k}$.
By comparing equations \eqref{eq_Fhat} and \eqref{eq_BesselJack}, we have 
\begin{equation}\label{eq_taufunctionisbessel}
    \begin{split}
        \hat{F}^{(1+\frac{z}{N})^{-1}}\left(\frac{1}{N\theta},\bfp(\vecx),\bfq(a),\left\{\left(-N\right)^{-k}\right\}_{k\ge 1}; \theta\right) = B_a(\vecx;\theta),
    \end{split}
\end{equation}
as an equality of analytic functions, for any $N\in\N$.
By Theorem~\ref{thm:CDmaintheorem}, we have the formal expansion:
\begin{multline}\label{eq_keyexpansion2}
    \ln \big( B_a(\vecx;\theta) \big) 
    = \sum_{d=1}^n \frac{\alpha^{d-1}}{d}\sum_{\MM\in C^{(\infty)}_{d,\bullet}\colon \MM\text{ normal}}(\alpha-1)^{\nu(\MM)}\prod_{f\in F(\MM)}\frac{p_{deg(f)}(\vecx)}{N}\prod_{v\in V_{0}(\MM)}\frac{p_{deg(v)}(a/N)}{N}\\
    \times N^{\ell\big(\pmb{\mu_1}(\MM)\big) + \ell\big(\pmb{\mu_2}(\MM)\big)-\big|\eta(\MM)\big|}\cdot f_{\eta(\MM)}\Big(\big\{(-1)^{-k}\big\}_{k\ge 1}\Big)+O\big(\|x\|^{n+1}\big).
\end{multline}
By Proposition~\ref{prop:euler}, the exponent of $N$ in the second line of \eqref{eq_keyexpansion2} is equal to $2-2g(\MM)$.
For a fixed $g(\MM)=g$, we have that $\ell(\eta(\MM))\le|\eta(\MM)|=\ell(\pmb{\mu_1}(\MM))+\ell(\pmb{\mu_2}(\MM))+2g-2<|\pmb{\mu_1}(\MM)|+|\pmb{\mu_2}(\MM)|+2g=2|\MM|+2g=2d+2g\le 2n+2g$.
As a result, if we set $k:=2n+2g$, then $\MM$ can be regarded as a $k$-constellation. In addition, by \cite[Theorem 3.1]{BD}, the total number of $k$-constellation of size $d$ is finite. We conclude that Eqn.~\eqref{eq_keyexpansion2} can be rewritten as Eqn.~\eqref{eq_keyexpansion}, such that for each fixed $g\in \frac{1}{2}\Z_{\ge 0}$, the sum $\sum_{\MM\in C^{(\infty)}_{d,\bullet}\colon\MM\text{ normal},\ g(\MM)=g}$ is finite.

It remains to check that when $N>\max(\alpha,1)\cdot n$, the sum over $g\in\frac{1}{2}\Z_{\ge 0}$ in \eqref{eq_keyexpansion} converges absolutely.
Using Eqn.~\eqref{eq_BesselJack} and the fact that $B_a(0^{N};\theta)=1$, one can expand $\ln\left( B_a(\vecx;\theta)\right)$ in terms of products of the form $\prod_{j=1}^{m}J_{\la_{j}}(a;\theta)J_{\la_j}(\vecx;\theta)$, for some partitions $\la_1,\dots,\la_m$, and subsequently, in terms of products of the form $p_{\mu_{1}}(\vecx)p_{\mu_{2}}(a)$, where $|\mu_{1}|=|\mu_{2}|$.
The coefficients in this expansion are rational functions of $N$ that admit a Taylor expansion at $N=\infty$, if $N>\max(\alpha,1)\cdot n$. By Eqn.~\eqref{eq_taufunctionisbessel}, such expansion agrees with 
\begin{multline*}
    \frac{\alpha^{d-1}}{d}\sum_{g\in \frac{1}{2}\Z_{\ge 0}} \sum_{\substack{\MM\in C^{(\infty)}_{d,\bullet}\colon \MM\text{ normal},\\ g(\MM)=g,\ \pmb{\mu_{1}}(\MM)=\mu_1,\ \pmb{\mu_2}(\MM)=\mu_2}}
    (\alpha-1)^{\nu(\MM)}
    N^{2-2g-\ell(\mu_1)-\ell(\mu_2)-d}\cdot f_{\eta(\MM)}\Big(\big\{(-1)^{-k}\big\}_{k\ge 1}\Big). 
\end{multline*}
Hence the latter expansion converges absolutely as well.
\end{proof}

\begin{remark}
    An alternative direct argument showing that the sum over $g\in\frac{1}{2}\Z_{\ge 0}$ in \eqref{eq_keyexpansion} converges uniformly over compact subsets, as a function of $(a,\vecx)\in\C^N\times\C^N$, as long as $N$ is very large, goes as follows.
    
    Let $\big(M(e,m)_{\la,\eta}\big)$ be the matrix of base change between elementary and monomial symmetric functions, i.e.,
    \[
    m_\eta = \sum_{\la\colon |\la|=|\eta|}{ M(e,m)_{\la,\eta} \cdot e_\la }.
    \]
    Then $f_\eta(x_1,x_2,\dots)=\sum_{\la\colon |\la|=|\eta|}{ M(e,m)_{\la,\eta} \prod_{i\ge 1}{x_{\la_i}} }$. In particular,
    \begin{equation}\label{eq:bound_f}
    \bigg| f_{\eta}\Big(\big\{(-1)^{-k}\big\}_{k\ge 1}\Big) \bigg| \le \sum_{\la\colon |\la|=|\eta|}{ \Big| M(e,m)_{\la,\eta}\Big| }.
    \end{equation}
    By \cite[Theorem 1]{KR}, for each $\la$, the absolute value $\big| M(e,m)_{\la,\eta}\big|$ is upper bounded by the number of Lyndon word sequences over an alphabet $A$ of cardinality $|A|=4$, and with lengths adding up to $|\eta|$. Since the number of Lyndon words of length $\ell$ over $A$ is upper bounded by $4^{\ell}$, then $\big| M(e,m)_{\la,\eta}\big|$ is upper bounded by $4^{|\eta|}$ times the number of compositions of $|\eta|$, which is $2^{|\eta|-1}<2^{|\eta|}$, thus, $\big| M(e,m)_{\la,\eta}\big|\le 8^{|\eta|}$. Plugging back into~\eqref{eq:bound_f} gives
    \[
    \bigg| f_{\eta}\Big(\big\{(-1)^{-k}\big\}_{k\ge 1}\Big) \bigg| \le 8^{|\eta|}\cdot p(|\eta|) \le (8e)^{|\eta|},
    \]
    where $p(m)$ denotes the number of integer partitions of $m$.
    On the other hand, one can verify from \cite[Definitions 3.6\,--\,3.8]{CD} that $0\le\nu(\MM)\le 2g(\MM)$, so $\big|(\alpha-1)^{\nu(\MM)}\big|\le \max(\alpha,1)^{2g(\MM)}$.
    Finally, as observed in the proof above, the inner sum over $\MM$ in \eqref{eq_keyexpansion} has a number of terms bounded by the number of $(2n+2g)$-constellations of size $d\le n$.
    By \cite[Theorem 3.1]{BD} and Stirling's formula, this number can be upper bounded by $n^{4n(n+g)}$.
    From the previous considerations, the sum of absolute values corresponding to the sum over $g\in\frac{1}{2}\Z_{\ge 0}$ in~\eqref{eq_keyexpansion} is upper bounded by $\sum_{g\in\frac{1}{2}\Z_{\ge 0}}{N^{2-2g}\cdot \max(\alpha,1)^{2g}\cdot (8e)^{2n+2g}\cdot n^{4n(n+g)}}$ times a constant independent of $g$.
    The last sum converges as long as $N>8e\cdot\max(\alpha,1)\cdot n^{2n}$.
    This verifies the Weierstrass M-test for the sum over $g\in\frac{1}{2}\Z_{\ge 0}$ in~\eqref{eq_keyexpansion} and shows that it converges uniformly on compact subsets.
\end{remark}

\subsection{Proof of the ``only if'' part of the main theorem}

Assume that $\{\mu_N\in\M_N\}_{N\ge 1}$ satisfies the fixed temperature LLN, meaning that
    \begin{equation}\label{assumption_LLN}
        \lim_{N\to\infty}{ \mathbb{E}_{\mu_N}\left[ \frac{1}{N^s}\prod_{j=1}^s{ p^{N}_{k_j}(a) } \right] } = \prod_{j=1}^s{m_{k_j}},
    \end{equation}
for all $s\in\Z_{\ge 1}$, $k_1,\dots,k_s\in\Z_{\ge 1}$, and a sequence of real numbers $\{m_k\}_{k\ge 1}$.
Assuming the expansions
\begin{equation}\label{eq:bessel_expansion}
\ln\Big( \E_{\mu_N}\big[ B_a(\vecx;\theta) \big] \Big) = \sum_{\la\colon|\la|\le N}{ a_\la^N p_\la(\vecx) }+O\big(\|x\|^{N+1}\big),
\end{equation}
we will prove that
\begin{equation}\label{eq:to_prove_1}
\lim_{N\to\infty}{\frac{a_{(d)}^N}{N}}\,\text{ exists and is finite, for all }d\in\Z_{\ge 1},
\end{equation}
and
\begin{equation}\label{eq:to_prove_2}
\lim_{N\to\infty}{\frac{a_\la^N}{N^{\ell(\la)}}}=0,\text{ whenever }\ell(\la)\ge 2.
\end{equation}

\smallskip

This will show that the sequence $\{\mu_N\}_{N\ge 1}$ is $\theta$-LLN-appropriate, and will conclude the proof of the ``only if'' part of the main theorem.
As before, we set $\alpha:=\theta^{-1}$.

\bigskip

\textbf{Step 1.} Let us first prove~\eqref{eq:to_prove_1}. Let $d\in\Z_{\ge 1}$ be arbitrary.
Assume that $N$ is large enough so that Theorem~\ref{thm:keyexpansion} can be applied.
By definition, $a^N_{(d)}$ is the coefficient of $p_d(\vecx)$ in the expansion of $\ln\E_{\mu_N}\big[B_a(\vecx,\theta)\big]$, but this is equal to the coefficient of $p_d(\vecx)$ in the expansion of $\E_{\mu_N}\big[\ln(B_a(\vecx,\theta))\big]$, which by Theorem~\ref{thm:keyexpansion}, equals
    \begin{multline}\label{eq:ad_expression}
       \frac{\alpha^{d-1}}{d}\sum_{g\in \frac{1}{2}\Z_{\ge 0}}\ \sum_{\MM\in C^{(\infty)}_{d,\bullet}\colon \MM\text{ normal},\, g(\MM)=g}
   (\alpha-1)^{\nu(\MM)}\cdot\E_{\mu_N}\left[\prod_{v\in V_{0}(\MM)}\frac{p_{deg(v)}(a/N)}{N}\right]\\
       \times N^{ 2-2g-\ell(\pmb{\mu_{1}(\MM)})}\cdot f_{\eta(\MM)}\left(\left\{\left(-1\right)^{-k}\right\}_{k\ge 1}\right).
    \end{multline}
   Recall that $\pmb{\mu_1}(\MM)$ is the partition with parts $\deg(v)$, $v\in V_0(\MM)$, and that $|\pmb{\mu_1}(\MM)|=|\MM|=d$, so the number of possible distinct expectation terms in the sum~\eqref{eq:ad_expression} is finite.
   Since each of these finitely many terms has a limit,
   \[
   \lim_{N\to\infty}{\E_{\mu_N}\left[\prod_{v\in V_0(\MM)} \frac{p_{deg(v)}(a/N)}{N}\right]} = \prod_{v\in V_{0}(\MM)}m_{deg(v)},
   \]
   by assumption~\eqref{assumption_LLN}, then all the expectations inside the sum~\eqref{eq:ad_expression} are of order $O(1)$; moreover, the exponent of $N$ inside the sum is at most $1$ and is attained when $g=0$ and $\ell(\pmb{\mu_1}(\MM))=1$.
   Since the genus is $g(\MM)=0$, then $\MM$ can be embedded in a sphere, so it is orientable, implying that $\nu(\MM)=0$. As a result,
   \begin{equation}\label{eq_freecumulant1}
        \lim_{N\rightarrow\infty}\frac{a^{N}_{(d)}}{N} = \frac{\alpha^{d-1}}{d}\sum_{\substack{\MM\in C^{(\infty)}_{d,\bullet}: \ \MM\ \text{normal},\\ \ell(\pmb{\mu_{1}}(\MM))=1,\ g(\MM)=0}}\ \prod_{v\in V_{0}(\MM)}{m_{deg(v)}}\cdot f_{\eta(\MM)}\left(\{(-1)^{k}\}_{k\ge 1}\right),
    \end{equation}
finishing the proof of~\eqref{eq:to_prove_1}.

\bigskip
 
\textbf{Step 2.} It remains to prove~\eqref{eq:to_prove_2}.
Let $\la$ be any partition with $\ell(\la)\ge 2$. Let $n:=|\la|$ and assume that $N$ is large enough so that Theorem~\ref{thm:keyexpansion} can be applied, i.e.~we have the expansion
\[
\ln\Big( B_a(\vecx;\theta) \Big) = \sum_{\nu:|\nu|\le n}{ b_\nu^N p_\nu(\vecx) } + O\big( \|x\|^{n+1} \big),
\]
where the coefficients $b_\nu^N$ are polynomials in $\alpha$, $N$ and the products $\prod_{v\in V_0(\MM)}{\frac{1}{N}p_{\deg(v)}(a/N)}$.
These products are of order $O(1)$, as $N\to\infty$, by the assumption~\eqref{assumption_LLN}.
Also, the analytic expansion~\eqref{eq_keyexpansion} shows that the leading asymptotic term in $\E_{\mu_N}\big[b_\nu^N\big]$ is of order $N^{2-\ell(\nu)}$ and comes from those constellations $\MM$ with genus $g(\MM)=0$.
In particular,
\begin{equation}\label{eq:def_c}
c_\nu := \lim_{N\to\infty}{\E_{\mu_N}\big[N^{\ell(\nu)-2}\cdot b_{\nu}^N\big]} \text{ exists and is finite, for all $\nu$ with $|\nu|\le n$}.
\end{equation}
More generally,
\begin{equation}\label{eq:finite_limit}
\lim_{N\to\infty}{\E_{\mu_N}\left[\prod_{i=1}^k \Big(N^{\ell(\nu^{(i)})-2}\cdot b_{\nu^{(i)}}^N\Big) \right]}
= \prod_{i=1}^k{c_{\nu^{(i)}}},
\end{equation}
for any sequence of partitions $\nu^{(1)},\dots,\nu^{(k)}$ with $|\nu^{(1)}|,\cdots,|\nu^{(k)}|\le n$.

\smallskip

Next, by expanding the exponential in the right hand side of
\begin{equation}\label{eq:exp_bs}
\ln\Big( \E_{\mu_N}\big[ B_a(\vecx;\theta) \big] \Big)
= \ln\left( \E_{\mu_N}\bigg[ e^{\ln{B_a(\vecx;\theta)}} \bigg] \right)
= \ln\left( \E_{\mu_N}\bigg[ e^{\sum_{\nu:|\nu|\le n}{ b_\nu^N p_\nu(\vecx) } + O\big( \|x\|^{n+1} \big)} \bigg] \right)
\end{equation}
and by comparing the result with~\eqref{eq:bessel_expansion}, we can find a formula for $a_\la^N$ in terms of the expected values $\Big\{\mathbb{E}\big[b_{\nu^{(1)}}^N\cdots b_{\nu^{(k)}}^N\big] \colon k\in\Z_{\ge 1},\, |\nu^{(1)}|+\dots+|\nu^{(k)}|\le n\Big\}$.
Such formula, together with~\eqref{eq:finite_limit}, will prove the desired limit~\eqref{eq:to_prove_2}.

\begin{example}
if $\la=(\la_1,\la_2)$ is a partition with $\ell(\la)=2$, then
\[
a_{(\la_1,\la_2)}^N = \E_{\mu_N}\Big[ b_{(\la_1,\la_2)}^N \Big] + \E_{\mu_N}\Big[ b_{(\la_1)}^N b_{(\la_2)}^N \Big] - \E_{\mu_N}\Big[ b_{(\la_1)}^N \Big] \cdot\E_{\mu_N}\Big[ b_{(\la_2)}^N \Big].
\]
Consequently, by~\eqref{eq:def_c}, we have
\[
N^{-2}a_{(\la_1,\la_2)}^N = O(N^{-2}) + \E_{\mu_N}\Big[ \big(N^{-1} b_{(\la_1)}^N\big) \big(N^{-1} b_{(\la_2)}^N\big) \Big] - \E_{\mu_N}\Big[ N^{-1} b_{(\la_1)}^N \Big] \cdot\E_{\mu_N}\Big[ N^{-1} b_{(\la_2)}^N \Big],
\]
and this converges to $c_1c_2 - c_1c_2=0$, as $N\to\infty$, by virtue of~\eqref{eq:finite_limit}.
\end{example}

As in the previous example, when $\ell(\la)\ge 2$, we can express $a_\la^N$ as a linear combination of products of expectations
\begin{equation}\label{eq:products}
\E_{\mu_N}\Big[ b_{\mu^{(1)}}^N\cdots b_{\mu^{(\ell_1)}}^N \Big]\cdot \E_{\mu_N}\Big[ b_{\mu^{(\ell_1+1)}}^N\cdots b_{\mu^{(\ell_1+\ell_2)}}^N \Big]\cdots
\E_{\mu_N}\Big[ b_{\mu^{(L-\ell_s+1)}}^N\dots b_{\mu^{(L)}}^N \Big],
\end{equation}
where $\mu^{(1)}\cup\cdots\cup\mu^{(L)}=\la$; note in particular that $\ell(\la)\ge L$, with equality if and only if each $\mu^{(i)}$ is a row partition.
Because of~\eqref{eq:def_c}-\eqref{eq:finite_limit}, it follows that this product is of order
\[
O\Big( N^{\sum_{i=1}^L\big[(2-\ell(\mu^{(i)}))\big]} \Big).
\]
Since $\ell(\mu^{(i)})\ge 1$, then $\sum_{i=1}^L\big[(2-\ell(\mu^{(i)}))\big]\le L\le\ell(\la)$, with equality if and only if each $\mu^{(i)}$ is a row partition.
This implies that the product~\eqref{eq:products} is of order $O\big(N^{\ell(\la)}\big)$, and if some $\mu^{(i)}$ has $\ell\big(\mu^{(i)}\big)\ge 2$, then~\eqref{eq:products} is of order $o\big(N^{\ell(\la)}\big)$.
As a result, the limit of $N^{-\ell(\la)}a_\la^N$, as $N\to\infty$, depends on $\big\{b_{(d)}^N\colon d\in\Z_{\ge 1}\big\}$, but not on $\big\{b_\nu^N\colon \ell(\nu)\ge 2\big\}$, so we can ignore the $b_\nu$'s with $\ell(\nu)\ge 2$.
As a result of this argument, together with equations~\eqref{eq:bessel_expansion} and~\eqref{eq:exp_bs}, we have that
\begin{equation}\label{eq:limit_a_proof}
\lim_{N\to\infty}{N^{-\ell(\la)} a_\la^N} = \lim_{N\to\infty}{N^{-\ell(\la)}}\big[p_\la(\vecx)\big]\ln\mathbb{E}_{\mu_N}\bigg[ e^{\sum_{d=1}^n{b_{(d)}^Np_d(\vecx)}} \bigg].
\end{equation}
The latter coefficient of $p_\la(\vecx)$ can be found by the moment-cumulant formula,
\begin{equation}\label{eq:mom_cum_proof}
\big[p_\la(\vecx)\big]\ln\mathbb{E}_{\mu_N}\bigg[ e^{\sum_{d=1}^n{b_{(d)}^Np_d(\vecx)}} \bigg] = \sum_{\pi\in\mathcal{P}(\ell(\la))}{ (-1)^{|\pi|-1}(|\pi|-1)!\cdot\prod_{B\in\pi}{\mathbb{E}_{\mu_N}\left[ \prod_{j\in B}{b^N_{(j)}} \right]} },
\end{equation}
where, on the right hand side, $\mathcal{P}(\ell(\la))$ denotes the collection of set partitions $\pi$ of $\{1,2,\dots,\ell(\la)\}$, and $|\pi|$ denotes the number of blocks of $\pi$.
Hence, by~\eqref{eq:finite_limit} applied to row partitions, and equalities~\eqref{eq:limit_a_proof}-\eqref{eq:mom_cum_proof}, we deduce the desired limit:
\begin{align*}
\lim_{N\to\infty}{N^{-\ell(\la)}a_\la^N}
= \sum_{\pi\in\mathcal{P}(\ell(\la))}{ (-1)^{|\pi|-1}(|\pi|-1)!\prod_{B\in\pi}{ \prod_{j\in B}{c_j} } }
&= \sum_{\pi\in\mathcal{P}(\ell(\la))}{ (-1)^{|\pi|-1}(|\pi|-1)!}\cdot\prod_{j=1}^{\ell(\la)}{c_j}\\
&= 0\cdot\prod_{j=1}^{\ell(\la)}{c_j}=0.
\end{align*}
We remark that the equality used above to go from the first to the second line is a well-known identity: in fact, it is the moment-cumulant formula for calculating the $\ell(\la)$-th cumulant of the delta mass $\delta_1$, since this measure has moments $m_k=1$ and classical cumulants $\kappa_\ell=\mathbf{1}_{\{\ell=1\}}$.

In conclusion, we have obtained the desired $\lim_{N\rightarrow\infty}{N^{-\ell(\la)}a^N_\la} = 0$, thus finishing the proof of~\eqref{eq:to_prove_2}, as well as \textbf{the ``only if'' part of the main theorem}.

\begin{remark}
    Eqn.~\eqref{eq_freecumulant1} furnishes the reverse formula of Eqn.~\eqref{eq_fixedtempmoment-cumulant}, in terms of constellations:
    \begin{equation}\label{eq_freecumulant2}
        \kappa_d = \sum_{\substack{\MM\in C^{(\infty)}_{d,\bullet}: \ \MM\ \text{normal},\\ \ell(\pmb{\mu_{1}}(\MM))=1,\ g(\MM)=0}}\ \prod_{v\in V_{0}(\MM)}{m_{deg(v)}}\cdot f_{\eta(\MM)}\left(\{(-1)^{k}\}_{k\ge 1}\right),\text{ for all }d\in\Z_{\ge 1}.
    \end{equation}
    Since constellations on closed orientable surfaces are in one-to-one correspondence with branched coverings of a sphere (see e.g.~\cite[Section 2]{CD}), this recovers a classical connection between free probability and geometry; see~\cite{BoGF} and the references therein.
\end{remark}

\section{Applications}\label{sec:application}

In this section, we give several applications of Theorem~\ref{thm:main}.

\subsection{$\theta$-additions and free convolution}

\begin{proof}[Proof of Theorem \ref{thm:beta_addition}]
    Let the power sum expansions of the logarithms of Bessel generating functions of $a(N)$ and $b(N)$ be
    \[
    \ln\big(G^{(1)}_{N,\theta}(\vecx)\big) = \sum_{\la\colon|\la|\le N}{a^N_\la[1]\,p_\la(\vecx)} + O\big(x^{N+1}\big),\quad \ln\big(G^{(2)}_{N,\theta}(\vecx)\big) = \sum_{\la\colon|\la|\le N}{a^N_\la[2]\,p_\la(\vecx)} + O\big(x^{N+1}\big),
    \]
    respectively.
    By the assumption of the theorem and the ``only if" part of Theorem~\ref{thm:main}, there exist two sequences of real numbers $\big\{\kappa_d^{(1)}\big\}_{d=1}^{\infty}$, $\big\{\kappa_d^{(2)}\big\}_{d=1}^{\infty}$, such that 
    \begin{align*}
        &\lim_{N\to\infty}{\frac{a_{(d)}^N[1]}{N}} = \frac{\theta^{1-d}\kappa_d^{(1)}}{d},\quad \lim_{N\to\infty}{\frac{a_{(d)}^N[2]}{N}} = \frac{\theta^{1-d}\kappa_d^{(2)}}{d},\quad d\ge 1,\\
        &\lim_{N\to\infty}{\frac{a_\la^N[1]}{N^{\ell(\la)}}} = \lim_{N\to\infty}{\frac{a_\la^N[2]}{N^{\ell(\la)}}}=0, \quad \ell(\la)\ge 2.
    \end{align*}
    Then, by the definition~\eqref{eq_thetaaddition2} of $\theta$-additions, the Bessel generating function of $a+_\theta b$ is $G^{(3)}_{N,\theta}(\vecx) = G^{(1)}_{N,\theta}(\vecx)\cdot G^{(2)}_{N,\theta}(\vecx)$, thus implying the expansion
    \[
    \ln\big(G^{(3)}_{N,\theta}(\vecx)\big) = \sum_{\la\colon|\la|\le N}{a^N_\la[3]\,p_\la(\vecx)} + O\big(x^{N+1}\big),
    \]
    where $a^N_\la[3] = a^N_\la[1] + a^N_\la[2]$, for all $\la$. Hence,
    \begin{align*}
        &\lim_{N\to\infty}{\frac{a_{(d)}^N[3]}{N}} = \frac{\theta^{1-d}\big(\kappa_d^{(1)}+\kappa_d^{(2)}\big)}{d}, \quad d\ge 1,\\
        &\lim_{N\to\infty}{\frac{a_\la^N[3]}{N^{\ell(\la)}}}=0, \quad \ell(\la)\ge 2.
    \end{align*}
    The result then follows from the ``if" part of Theorem \ref{thm:main}.
\end{proof}

\subsection{$\theta$-Dyson Brownian motion}

The Dyson Brownian motion, introduced in~\cite{Dy}, is an $N$-dimensional diffusion that can be viewed as the eigenvalues of a process $(M_{N}(t))_{t\ge 0}$ of $N\times N$ complex Hermitian matrices whose entries are i.i.d.~Brownian motions. The diffusion is generalized (from $\theta=1$) to a general $\theta>0$ by means of the following system of SDEs
\begin{equation}\label{eq_dbm}
    \d W_i(t)=\theta\sum_{j\ne i}\frac{\d t}{W_i(t)-W_j(t)} + \d B_i(t),\quad i=1,2,...,N,
\end{equation}
where $B_1(t),\dots,B_N(t)$ are $N$ independent standard Brownian motions.

\begin{definition}
    Fix $\theta>0$. Let $a=(a_1\le ...\le a_N)$ be a random $N$-tuple with exponentially decaying distribution. The $\theta$-Dyson Brownian motion ($\theta$-DBM for short) starting at $a$ is the unique strong solution $W(t) = \big(W_1(t)\le ...\le W_N(t)\big)$, $t\ge 0$, of the system of SDEs in Eqn.~\eqref{eq_dbm}, with initial condition $W(0)=a$.
\end{definition}

If the solution mentioned in the previous definition exists, then it lives in the closed Weyl chamber $\mathcal{W}_N := \{(x_1,\dots,x_N)\in\R^N \mid x_1\le x_2\le\dots\le x_N\}$.
When $\theta<\frac{1}{2}$, some issues arise from the definition above, since the particles collide after a finite time and therefore the solution of Eqn.~\eqref{eq_dbm} blows up; see~\cite{CL}.
However, $(W(t))_{t\ge 0}$ can be alternatively defined as an $N$-dimensional Markov process with a certain explicit transition kernel, and this definition is valid for all $\theta>0$; see e.g.~\cite[Eqn. (23)]{GXZ}.
The $\theta$-DBM has various connections to models in the KPZ universality class, to random geometry, and to random matrices; see e.g.~\cite{CH,DOV,Ca,EY} and references therein.

Under the trivial initial condition $a=(0^N)$, the fixed-time distribution of $W(t)$, for any given $t>0$, is known to be an $N$-dimensional Gaussian $(2\theta)$-ensemble, so in particular, the distribution of $W(t)$ is exponentially decaying.
The following statement identifies $(W(t))_{t\ge 0}$, starting at $a$, with the $\theta$-addition of $a$ and the $\theta$-DBM starting at $(0^N)$.

%Add reference about $G\beta E$?

%$$\frac{1}{Z_{N,\beta}}\prod_{1\le i<j\le N}|x_{i}-x_{j}|^{\beta}\exp\left(-\frac{1}{2T}\sum_{i=1}^{N}x_{i}^{2}\right).$$

\begin{proposition}[Lemma~3.8 in \cite{GXZ}]\label{prop:dbmbgf}
    Let $W(t)$, $t\ge 0$, be an $N$-dimensional $\theta$-DBM starting at $a\in\mathcal{W}_N$. For any $t\ge 0$, the Bessel generating function of the distribution of $W(t)$ is equal to
    \begin{equation}
        \E\Big[B_{W(t)}(x_{1},...,x_{N};\theta)\Big] = \E\big[B_a(x_{1},...,x_{N};\theta)\big]\cdot \exp\left(\frac{t}{2}\sum_{i=1}^{N}x_{i}^{2}\right).
    \end{equation}
\end{proposition}

%Consider the empirical measure of $W(t)$ as $N\rightarrow\infty$. We prove that under proper scaling, the global limit of $\mu[W(t)]$ exists, and is independent of $\theta$.

\begin{theorem}\label{thm:beta_dyson}
   Let $\theta>0$ be fixed and let $a(N) = \big(a_1(N)\le ...\le a_N(N)\big)$, $N=1,2,\dots$, be a sequence of random $N$-tuples in $\mathcal{W}_N$ with exponentially decaying distributions. Assume that there exists a probability measure $\mu$ with compact support such that
    \[
    \lim_{N\rightarrow\infty}\mu[a(N)]\stackrel{m}{=}\mu,
    \]
    where $\mu[a(N)]$ is defined in Eqn.~\eqref{eq:empirical_measure}.
     Let $\big(W_N(t)\big)_{t\ge 0}$ be the $\theta$-DBM starting at $a(N)$. Then for any fixed $T>0$, we have
    \[
    \lim_{N\rightarrow\infty}\mu\Big[W_N(\theta^{-1} TN)\Big]\stackrel{m}{=}\mu\boxplus\mu_{sc}^{(T)},
    \]
    weakly, in probability, where $\mu_{sc}^{(T)}$ is the semicircle law with density
    \[
    \frac{\d\mu_{sc}^{(T)}}{\d x}(x) = \mathbf{1}_{\big[ -2\sqrt{T}, 2\sqrt{T} \big]}(x)\cdot\frac{1}{2\pi T}\sqrt{4T-x^{2}}.%\ -2\sqrt{T}\le x\le 2\sqrt{T}.
    \]
\end{theorem}
\begin{proof}
    A simple calculation shows that $\mu_{sc}^{(T)}$ has free cumulants $\kappa_n=T\cdot\mathbf{1}_{\{n=2\}}$, for all $n\ge 1$. Then the result is an immediate consequence of Theorem~\ref{thm:main} and Proposition~\ref{prop:dbmbgf}.
\end{proof}

\subsection{$\theta$-corners and free projection}

\begin{proposition}\label{prop:bgfcorners}
    Given a random $N$-tuple $a=(a_1<\dots<a_N)$ and integer $1\le M<N$, we have
\[
\E\left[B_{corner_{M}^{N}(a)}(x_{1},\dots,x_{M};\theta)\right] = \E\Big[B_a\big(x_1,\dots,x_M,0^{N-M};\theta\big)\Big].
\]
\end{proposition}
\begin{proof}
    This is an immediate consequence of the integral representation in Eqn.~\eqref{eq_besselbranching}.
\end{proof}

\begin{proof}[Proof of Theorem \ref{thm:beta_projection}]
     Denote the Bessel generating functions of $a(N)$ and $corner^N_M(a)$, respectively, by
     \[
     G_{N,\theta}(\vecx) = \sum_{\la\colon|\la|\le N}a^N_\la p_\la(\vecx)+O(\|x\|^{N+1}),\qquad
     G^{(N)}_{M,\theta}(\vecx)=\sum_{\la\colon|\la|\le M} a^N_\la[M]\,p_\la(\vecx)+O(\|x\|^{M+1}).
     \]
     By the ``only if'' part of Theorem \ref{thm:main}, there exist real numbers $\kappa_1,\kappa_2,\dots$ such that
    \begin{equation*}
    \lim_{N\to\infty}{\frac{a_{(d)}^N}{N}} = \frac{\theta^{1-d}\kappa_d}{d},\quad\text{for all }d\ge 1,\qquad\quad
    \lim_{N\to\infty}{\frac{a_\la^N}{N^{\ell(\la)}}}=0,\quad\text{if }\ell(\la)\ge 2.
    \end{equation*}
    Then by definition of $\theta$-corners, we have $a^N_\la[M] = a^N_\la$, for all $\la$ such that $|\la|\le M$.
    Then, if $\alpha\in (0,1)$ and $M = \lfloor \alpha N\rfloor$, we have
    \begin{equation*}
        \lim_{N\to\infty}{\frac{a_{(d)}^N[M]}{M}} = \frac{\theta^{1-d}\kappa_d}{\alpha d},\quad\text{for all } d\ge 1,\qquad\quad
        \lim_{N\to\infty}{\frac{a_\la^N[M]}{M^{\ell(\la)}}}=0,\quad\text{if }\ell(\la)\ge 2.
    \end{equation*}
    Next, by the ``if'' part of Theorem~\ref{thm:main}, we have
    $$\lim_{N\rightarrow\infty}corner^{N}_{\lfloor \alpha N\rfloor}(a)\stackrel{m}{=}\big\{m_k^{(\alpha)}\big\}_{k=1}^{\infty},$$
    where the quantities $\big\{m_k^{(\alpha)}\big\}_{k=1}^{\infty}$ are derived from the sequence $\{\frac{1}{\alpha}\kappa_d\}_{d=1}^{\infty}$ by means of Eqn.~\eqref{eq_fixedtempmoment-cumulant}.
    In particular, Eqn.~\eqref{eq_fixedtempmoment-cumulant} shows that $\big|m_k^{(\alpha)}\big|\le\alpha^{-k}\cdot|m_k|$, for all $k\ge 1$. As a result, the sequence $\big\{m_k^{(\alpha)}\big\}_{k=1}^\infty$ satisfies the Carleman's condition for the moment problem (see e.g \cite{Sc}), and therefore  uniquely determines the limiting measure. 
\end{proof}

\begin{remark}
    The convergence in the sense of moments stated in Theorems \ref{thm:beta_projection} and \ref{thm:beta_dyson} also holds in the sense of weak convergence, in probability. This is because with certain tightness conditions that hold for the empirical measures in these two theorems, the two types of convergence are equivalent.
\end{remark}

\end{document}